\documentclass[a4paper]{amsart}

\usepackage{amssymb}
\usepackage[all]{xy}
\usepackage{amsmath, amsfonts, amsthm , amssymb, amscd}
\usepackage{hyperref}
\usepackage{booktabs}
\usepackage{todonotes}
\usepackage{mathtools}
\usepackage{enumitem}
\usepackage{romannum}
\usepackage{cleveref}
\usepackage{tikz}
\usetikzlibrary{decorations.pathmorphing}
\newtheorem*{theorem*}{Theorem A}
\newtheorem*{theorem**}{Theorem B}

\newtheorem{theorem}{Theorem}[section]

\newtheorem{proposition}[theorem]{Proposition}
\newtheorem{lemma}[theorem]{Lemma}
\newtheorem{corollary}[theorem]{Corollary}

\theoremstyle{definition}
\newtheorem{definition}[theorem]{Definition}

\theoremstyle{remark}
\newtheorem{remark}[theorem]{Remark}

\newtheorem{problem}[theorem]{Problem}

\renewcommand{\span}{\mathop{\mathrm{span}}}

\renewcommand{\Im}{\mathop{\mathrm{Im}}}

\newcommand{\link}{\ensuremath{\mathrm{Lk}}}
\DeclareMathOperator{\Lk}{Lk}

\newcommand{\Id}{\ensuremath{\mathrm{Id}}}

\newcommand{\N}{\ensuremath{\mathbb{N}}}
\newcommand{\Z}{\ensuremath{\mathbb{Z}}}

\newcommand{\F}{\ensuremath{\mathbb{F}}}

\renewcommand{\k}{\ensuremath{\mathbf{k}}}
\newcommand{\xkn}{\ensuremath{X(\mathbf{k}^n)}}
\newcommand{\kkn}{\ensuremath{K(\mathbf{k}^n)}}

\newcommand{\sign}{\ensuremath{\mathrm{sign}}}
\newcommand{\lra}{\longrightarrow}

\title{Universal simplicial complexes inspired by toric topology}

\author{Djordje Barali\'{c}, Jelena Grbi\'{c}, Ale\v s Vavpeti\v c, Aleksandar Vu\v{c}i\'{c}}
\address{ \scriptsize{Mathematical Institute SASA, Belgrade, Serbia }}
\email{djbaralic@mi.sanu.ac.rs}
\address{\scriptsize{School of Mathematics, University of Southampton,
         Southampton, SO17 1BJ, United Kingdom}}
\email{J.Grbic@soton.ac.uk}
\address{Faculty of Mathematics and Physics, University of Ljubljana, Slovenia and Institute of Mathematics, Physics and Mechanics, Ljubljana, Slovenia}
\email{ales.vavpetic@fmf.uni-lj.si}
\address{\scriptsize{University of Belgrade, Faculty of Mathematics, Belgrade, Serbia}}
\email{avucic@matf.bg.ac.rs}


\date{}

\begin{document}

\pagenumbering{arabic}

\begin{abstract}

Let $\mathbf{k}$ be the field $\mathbb{F}_p$  or the ring $\mathbb{Z}$. We study combinatorial and topological properties of the universal simplicial complexes $\xkn$ and $\kkn$ whose simplices are certain unimodular subsets of $\k^n$. 
As a main result we show that  $\xkn$, $\kkn$ and the links of their simplicies are homotopy equivalent to a wedge of spheres specifying the exact number of spheres in the corresponding wedge decompositions. This is a generalisation of Davis and Januszkiewicz's result that $K(\Z^n)$ and $K(\F_2^n)$ are $(n-2)$-connected simplicial complexes.  

We discuss applications of these universal simplicial complexes to toric topology and number theory.

\end{abstract}

\maketitle

\section{Introduction}

An abstract simplicial complex on a set $\mathcal{S}$ is a collection $K$ of subsets of $\mathcal{S}$ closed under taking subsets. We assume that $\mathcal S\subset K$. and refer to $S$ as the vertex set of $K$. The elements of $K$ are called simplices. A simplex of $K$ which is not a proper subset of some other simplex in $K$ is called a \textit{maximal simplex} or a \textit{facet}.

The dimension of a simplex $\sigma\in K$ is $\dim \sigma =| \sigma| -1$. The dimension of a simplicial complex $K$ is the maximal dimension of its simplices. 
A simplicial complex $K$ is \emph{pure} if all its maximal simplices are of the maximal dimension. 
A simplicial complex $K$ is \textit{finite} if the number of its simplices is finite, otherwise it is \textit{infinite}.

The link of  $\sigma \subset S$ in
$K$ is the subcomplex
\begin{equation*}
\mathrm{\link}_K (\sigma)  =  \{\tau \text{ }|\text{ } \sigma\cup\tau\in K,
\sigma\cap\tau=\emptyset\}. \end{equation*}
If $K$ is assumed, the link $\Lk_K(\sigma)$ will be denoted by $\Lk(\sigma )$.

\begin{definition}
Let $\mathbf{k}$ be the field $\mathbb{F}_p$
or the ring $\mathbb{Z}$. A set $\{ v_1, \ldots , v_m \}$ of elements in $\mathbf{k}^n$ is called \emph{unimodular} if $\span \{v_1,\ldots ,v_m\}$ is a direct summand of $\mathbf{k}^n$ of dimension $m$. The simplicial complex $X(\mathbf{k}^n)$ on unimodular vertices $v_i\in \k^n$ consists of all unimodular subsets of $\k^n$.  Note that a subset of a unimodular set is itself unimodular.
\end{definition}

Initially the simplicial complex $\xkn$ appeared independently in algebraic $K$-theory in the work of van der Kallen~\cite{MR586429} and in combinatorics in the work of Stanley~\cite{MR593545}.
Closely related to the simplicial complex $\xkn$ is the Tits building $T(V)$ for an $n$-dimensional vector space $V$. An $m$-simplex of $T(V)$ is a chain $W_0<\ldots<W_m$ of proper subspaces. Studying finite generation of the higher algebraic $K$-groups, Quillen~\cite{MR2655185} showed that $T(V)$ has the homotopy type of a wedge of $(n-2)$ spheres. This result further prompted the study of the interesting homology group $H_{n-2}(T(V))$. This is a $GL(V)$-module called the Steinberg module. Further on, van der Kallen~\cite{MR586429} proved that the link of an $m$-simplex of $T(V)$ is $(n-m-3)$-connected.

Another simplicial complex closely related to $ X(\mathbf{k}^n)$ is the \textit{simplicial complex} $\kkn$.

Define a line through the origin to be a 1-dimensional unimodular subspace of $\k^n$, that is, a submodule of rank $1$ which is a direct summand of $\k^n$ and denote by $P\k^n$ the set of all such lines in $\k^n$.
\begin{definition}
The simplicial complex $\kkn$ on the vertex set  $P\k^n$ is defined such that its $m$-simplex
 is a set of lines $\{l_1, \dots, l_{m+1}\}$, $l_i\in P\k^n$ which span an $(m+1)$-dimensional unimodular subspace of $\mathbf{k}^n$.
\end{definition}

The simplicial complexes $\xkn$ and $\kkn$ are related through the map which sends a nonzero
unimodular element of $\k^n$ into a line generated by it. The simplicial complex $\kkn$ is a retract of the simplicial complex $\xkn$. 



The simplicial complexes $K(\F_2^n)$ and $K(\mathbb{Z}^n)$ are important objects in toric topology since they are universal simplicial complexes classifying small covers for $\k=\F_2$ and quasitoric manifolds for $\k=\Z$, see~\cite{DJ}.

Quasitoric manifolds can be equipped with an additional structure called omniorientation, see \cite{MR1897064}. As an analogue of the universal property of $K(\Z^n)$ and $K(\F_2^n)$, we show, as Proposition~\ref{uni}, that the complex $X(\Z^n)$ is a universal complex classifying omnioriented quasitoric manifolds $M^{2n}(P,\Lambda)$ with positive orientation on the polytope~$P$.

The goal of this paper is a combinatorial and topological study of the simplicial complexes $\xkn$ and $\kkn$ and the links of their simplices and their applications to related mathematical areas. After discussing their universal properties in Section~\ref{universal}, we focus on determining the $f$-vectors of $X (\mathbb{F}_p^n)$, $K(\mathbb{F}_p^n)$ and the $f$-vectors of the links of their simplices in~\Cref{month}, which will be used to identify the homotopy type of $X(\F_p^n), K(\F_p^n)$ and of the links of their simplicies as a wedge of spheres. As the universal complexes $X(\F_p^n)$ and $K(\F_p^n)$ are finite matroids, they are shellable complexes and thus by~\cite[Theorem~1.3]{MR0744856} they are homotopy equivalent to a wedge of $(n-1)$-spheres. We improve on that result by specifying exactly how many $(n-1)$-dimensional spheres are in the wedge in \Cref{mainKp}.

Davis and Januszkiewicz~\cite{DJ} proved that both complexes $K(\F_2^n)$ and $K(\mathbb{Z}^n)$ are $(n-2)$-connected and that the link of an $m$-simplex is $(n-m-3)$-connected.
However, there is a gap in their proof of the connectivity of $K(\mathbb{Z}^n)$ as they used van der Kallen result~\cite[Theorem 2.6]{MR586429} which guarantees the connectivity of $K(\mathbb{Z}^n)$ and its links only up to dimension $n-3$ not as claimed in \cite[Theorem~2.2]{DJ} up to dimension $n-2$. We obtain a stronger result, as~\Cref{main:z}, showing that $K(\mathbb{Z}^n)$ and its links of their simplices are homotopy equivalent to a wedge of spheres.

 Our main result is that the universal complexes  $X(\Z^n)$ and $K(\Z^n)$  have homotopy type of a wedge of countably infinitely many spheres.
 
 \begin{theorem}
\label{main:z}

\begin{itemize}

\item [$(a)$] For $n\geq 2$, the simplicial complex $X(\mathbb{Z}^n)$ is homotopy equivalent to a countably infinite wedge of $(n-1)$-spheres.

\item [(b)] For $n\geq 2$ and $m \leq n-2$, the link of an $m$-simplex in $X (\mathbb{Z}^n)$ is homotopy equivalent to a countably infinite wedge of $(n-m-2)$-spheres.

\item [(c)] For $n\geq 2$, the simplicial complex $K(\mathbb{Z}^n)$ is homotopy equivalent to a countably infinite wedge of $(n-1)$-spheres.

\item [(d)] For $n\geq 2$ and or $m\leq n-2$, the link of an $m$-simplex in $K
(\mathbb{Z}^n)$ is homotopy equivalent to a countably infinite wedge of $(n-m-2)$-spheres.

\end{itemize}
\end{theorem}

In Section 4 we give the proof of~\Cref{main:z} which together with~\Cref{mainKp} and ~\Cref{main2}, generalise the results of Davis and Januszkiewicz, and van der Kallen.

For an infinitely generated algebra there is no algebraic notion of Cohen-Macaulayness. For a finite simplicial complex $K$, by the Reisner criterion, Cohen-Macaulayness of $\k[K]$  is equivalent to the reduced homology groups of the link of all simplices $\sigma$ in $K$ being trivial except in top degree. This combinatorial criterion lends itself to a generalisation of Cohen-Macaulyness to infinite simplicial complexes. Using this as the definition of Cohen-Macauly for finite dimensional infinite simplicial complexes,~\Cref{main:z} implies the following statement.

\begin{corollary}\label{corch}
The universal complexes  $X(\Z^n)$ and $K(\Z^n)$ are Cohen-Macaulay.
\end{corollary}


An important application of simplicial complexes $K(\F_2^n)$ and $K(\Z^n)$ is to the study of the Buchstaber invariant of simplicial complexes. The known upper and lower bound estimates of the Buchstaber invariants using topological and combinatorial properties of these complexes are obtained in~\cite{Ayze} and~\cite{MR3482595}.

In Section~\ref{sec:invariant}, we introduce the Buchstaber invariant of a simplicial complex $K$ over $\F_p$, where $p$ is a prime and relate it to the Buchstaber invariant of  the universal complexes $\xkn$ and $\kkn$. In doing so we consider nondegenerate simplicial maps into those universal complexes, their numerical invariants and the {$f\text{-vectors}$} of $K(\F_p^n)$. As a result we obtain an analogue of Ayzenberg's formula~\cite{Ayze1} for a graph.

Beside toric topology, we relate the universal simplicial complexes $X(\F_p^n)$ and its $f$-vectors to number theoretical problem about divisibility.
As an application, we answer Bahargava's question~\cite{Bhar} on combinatorial interpretations of certain generalised factorial function.

\section{A universal property of the simplicial complex $X(\Z^n)$}
\label{universal}

Our study of the simplical complexes $\xkn$ and $\kkn$ has been motivated by an important role $K(\Z^n)$ and $K(\F_2^n)$ play in toric topology.

Let $G_d$ be $\Z_2$ when $d=1$ and $S^1$ when $d=2$. Spaces with $G_d^n$-actions  are objects studied in toric topology. A special attention is given to $dn$-manifolds with a locally standard $G_d^n$-action having an $n$-dimensional simple polytope $P^n$ as the orbit space. Those manifolds are usually referred to as small covers for $d=1$ and quasitoric manifolds for $d=2$. They are topological generalisations of toric and real toric varieties. 
Recall that  by \cite[Construction~5.12]{MR1897064} the quasitoric manifolds and small covers are determined by a simple polytope $P^n$ and a characteristic map $\lambda\colon \mathcal{F}\rightarrow \pi_{d-1}(G_d^n)$, where $\mathcal{F}$ is the set of facets of $P$. In the case of quasitoric manifolds for each facet $F$ of $P^n$  the vector $\lambda (F)$ is determined up to sign. The characteristic map $\lambda$ satisfies a nondegeneracy condition, that is, $\{\lambda (F_{i_1}), \dots, \lambda (F_{i_k})\}$ is a unimodular set whenever the intersection of facets $F_{i_1}$, $\dots$, $F_{i_k}$ is nonempty. Let $\partial P^*$ be a simplicial complex which is the boundary of the dual to $P$. Then there is a simplicial map $\bar{\lambda}\colon \partial P^*\rightarrow K (\pi_{d-1}(G^n_d))$  defined on the vertices of $\partial P^*$ by $\lambda$ such that a vertex $v$ of $\partial P^*$ dual to the facet $F_v$ of $P$ is sent to the line $l (\lambda (F_v))$. The map $\lambda$ is a characteristic map if and only if $\bar{\lambda}$ is a nondegenerate simplicial map. Therefore, the set of equivalences classes of $G_d^n$-manifolds over $P$ is in bijection with the set of nondegenarate simplicial maps $f\colon \partial P^*\rightarrow K(\pi_{d-1}(G^n_d))$ modulo natural action of $\mathrm{Aut} (\pi_{d-1}(G_d^n))$. This universal property of $K(\pi_{d-1}(G^n_d))$ further motivates the combinatorial and topological study of $K(\pi_{d-1}(G^n_d))$.

Using omniorientation, Buchstaber, Panov and Ray~\cite{MR2337880} constructed a quasitoric representative for every complex cobordism class $\Omega^U$.  A choice of omniorientation on a given quasitoric manifold $M^{2n}$ over $P^n$ is equivalent to a choice of orientation on $P^n$ together with a choice of facet vectors $\lambda (F_{i})$. 
More precisely, recall that a pair $(P,\Lambda )$ is called a combinatorial quasitoric pair if $P$ is an oriented combinatorial simple $n$-polytope with $m$ facets and $\Lambda$ is an $n\times m $ integer matrix with the property that for each vertex $v$ appropriate minor $\Lambda_v$ of $\Lambda$, consisting of columns whose facets contain $v$, is invertible. Two combinatorial quasitoric pairs $(P,\Lambda )$ and $(P',\Lambda ')$ are equivalent if $P=P'$ with the same orientation and $ \Lambda = \Phi \Lambda '$ for some $n \times n$ integer matrix $\Phi $ with determinant 1.

Buchstaber, Panov and Ray~\cite{MR2337880} showed that there is a 1-1 correspondence between equivalence classes of omnioriented quasitoric manifolds and combinatorial quasitoric pairs.  
We use this combinatorial description of omnioriented quasitoric manifolds to describe a universal property of $X(\Z^n)$. 
\begin{proposition}
\label{uni}
There is a 1-1 correspondence between equivalence classes of omnioriented quasitoric manifolds with positive orientation on $P$ and nondegenerate simplicial maps 
\[
f\colon\partial P^* \lra X(\Z^n)
\]
 modulo natural action of $\mathrm{Aut} (\mathbb{Z}^n)$.
\end{proposition}

\begin{proof}
For a combinatorial quasitoric pair $(P,\Lambda )$, where $P$ is an oriented combinatorial simple $n$-polytope with $m$ facets and $\Lambda$ is an $n\times m $ integer matrix with the property that for each vertex $v$ appropriate minor $\Lambda_v$ of $\Lambda$ is invertible, define a map $f\colon\partial P^*\lra X(\Z^n)$ on vertices $v_i$ of $\partial P^*$ to be equal to the $i${th} column of $\Lambda$. The condition that $\Lambda_v$ is invertible for each vertex $v$ of $P$ is equivalent to $f$ being nondegenerate.

On the other hand, for a nondegenerate simplicial map $f\colon \partial P^*\lra X(\Z^n)$ we construct $\Lambda$ by placing vectors $f(v_i)$ as appropriate columns and thus obtain a combinatorial quasitoric pair $(P,\Lambda)$.
\end{proof}

As a side remark, to avoid possible confusion, let us point out that Davis and Januszkiewicz in~\cite{DJ} denoted by $X(\mathbf{k}^n)$ the barycentric subdivision of simplicial complex of all unimodular subsets of $\k^n$.

\section{The homotopy type of $K(\F_p^n)$, $X(\F_p^n)$ and their links}

Matroids are combinatorial objects which generalise the notion of linear independence in vector spaces.
A finite simplicial complex $M$ is called \emph{matroid} if for any two simplices $\sigma$, $\tau\in M$ such that $|\sigma|>|\tau|$, then for some $v\in \sigma\setminus \tau$ the set $\tau\cup \{v\}\in M$. Simplices of a matroid are called the independents sets and the defining property is called the independent set exchange property. 

\begin{lemma}
The simplicial complexes $X(\F_p^n)$ and $K(\F_p^n)$ are finite matroids.
\end{lemma}
\begin{proof}Let $\sigma=\{v_{i_1}, \dots, v_{i_l}\}$ and  $\tau=\{v_{j_1}, \dots, v_{j_m}\}$ be two simplices in $X(\F_p^n)$  such that $l>m$. As the dimension of a simplex in $X(\F_p^n)$ is one less than the dimension of the span of its vertices, there exists $v\in \sigma$ not lying in the span of the vertices in $\tau$. By the definition of $X(\F_p^n)$, for such  $v$, $\tau \cup \{v\}\in X(\F_p^n)$. Analogously, it follows readily that $K(\F_p^n)$  also satisfies the independent set exchange property and therefore $K(\F_p^n)$ is a matroid.
\end{proof}

A simplicial complex is \textit{shellable} if its facets can be arranged in linear order $F_1,F_2,\ldots,F_q,\ldots$ in such a way that $(\cup_{i=1}^{k-1}F_i)\cap F_k$ is pure and $(\dim F_k-1)$-dimensional for all $k = 2,\ldots,q,\ldots$. Such an ordering of facets is called a \textit{shelling order} or \textit{shelling}. 
For a shelling on a finite dimensional pure simplicial complex $K$ of at most countable cardinality, we define the \textit{restriction} of a facet $F_k$ by 
\[
r(F_k)=\{v\in F_k\left|\right. F_k\setminus \{v\} \subset (\cup_{i=1}^{k-1}\overline{F}_i)\},
\] 
where for $\sigma\in K$, by $\overline{\sigma}$ we assume the complex $\{\tau |\tau\subseteq \sigma\}$. A simplicial complex $K$ is \textit{shellable of characteristic} $h$ if for some shelling and corresponding restriction
\[
h=\left| \{F\in K\left|\right. F \text{\,is a facet} \text{\,such that\,}  r(F)=F\}\right|.
\]
Bj\" orner proved~\cite[Theorem~1.3]{MR0744856} that a pure $n$-dimensional shellable complex $K$ of at most countable cardinality has the homotopy type of a wedge of $h$ spheres $S^n$, where $h$ is the characteristic of shelling. Finite matroids are shellable complexes and have the homotopy type of a wedge of spheres, see~\cite{MR1165544} and~\cite[Theorem~1.3]{MR0744856}. For more information on matroid theory, we refer readers to~\cite{MR1207587} and~\cite{MR1226888}. Since the simplicial complexes $X(\F_p^n)$ and $K(\F_p^n)$ are natural examples of finite matroids they have the homotopy type of a wedge of spheres $S^{n-1}$. To determined the exact number of spheres in the wedge we calculate the $f$-vectors of  $X(\F_p^n)$ and $K(\F_p^n)$.

\subsection{$f$-vectors of $X(\F_p^n)$ and $K(\F_p^n)$}
The \emph{$f$-vector} of an $(n-1)$-dimensional simplicial complex
$K$ is the integer vector 
$$
\textbf{f}
(K)=\left(f_{-1}, f_0, f_1, \dots, f_{n-1}\right)
$$ 
where
$f_{-1}=1$ and $f_i=f_i (K)$ denotes the number of $i$-faces
of $K$ for all $i\geq0$.
The \textit{$f$-polynomial} of  an $(n-1)$-dimensional
simplicial complex $K$ is 
\[
\textbf{f} (t)=t^n+f_0
t^{n-1}+\dots+f_{n-1}.
\]
The \textit{Euler characteristic} of a simplicial complex $K$ is given by
\begin{equation}
\label{euler} \chi (K)=f_0-f_1+\cdots+(-1)^{n-1}
f_{n-1}=(-1)^{n-1} \textbf{f} (-1)+1.
\end{equation}

The $f$-vector of $X(\F_p^n)$ and $K(\F_p^n)$ and of the
links of their simplices can be explicitly described in terms of
$p$ and $n$. The $f$-vectors carry important combinatorial
information of $X(\F_p^n)$ and $K(\F_p^n)$ and knowing them helps 
to determine not only the homotopy type of
the universal complexes $X(\F_p^n)$ and $K(\F_p^n)$ but also the Buchstaber invariant of simplicial complexes.

Since the simplicial complexes $X(\F_p^n)$ and $K(\F_p^n)$ are
given as the sets of unimodular subsets of $\mathbb F^n_p$, they are
clearly  pure, $(n-1)$-dimensional and come with a transitive and simplicial action of $GL(n, \mathbb F_p)$. Therefore to find the $f$-vector of
$X(\F_p^n)$ and $K(\F_p^n)$ and of the links of their simplices it is enough to compute the number of linearly independent
$i$-tuples in $\mathbb F^n_p$, which is a standard exercise in linear algebra. For example, in the case of $X(\F_p^n)$, the calculation reduces to solving the recurrence relation $(i+2)f_{i+1}(X(\F_p^n))=(p^n-p^{i+1})f_i(X(\F_p^n))$.  

\begin{lemma}
\label{month}

\begin{enumerate}[label={(\alph*)},ref={\thelemma~(\alph*)}]
\item  \label{month1} The $f$-vector of $X (\mathbb{F}_p^n)$ is
given by
\begin{equation*}
 f_{i} (X(\mathbb{F}_p^n))= \frac{(p^n-p^{i}) \cdots (p^n-p^0)}{(i+1)!}
\end{equation*}
for $0\leq i\leq n-1$ and $f_{-1}=1$.

\item  \label{month2}The $f$-vector of $K (\mathbb{F}_p^n)$ is
given by
\begin{equation*}
 f_{i} (K (\mathbb{F}_p^n))= \frac{(p^n-p^{i}) \cdots (p^n-p^0)}{(p-1)^{i+1} (i+1)!}
\end{equation*}
for $0\leq i\leq n-1$ and $f_{-1}=1$.

\item  \label{month3}Let $\sigma $ be an $m$-simplex in
$X(\mathbb{F}_p^n)$. The $f$-vector of
$\Lk_{X(\mathbb{F}_p^n)}(\sigma)$ is given by
\begin{equation*}
f_i(\Lk_{X (\mathbb{F}_p^n)}(\sigma)) =
\frac{(p^n-p^{m+1})\cdots (p^n-p^{m+i+1})}{(i+1)!}
\end{equation*}
for $0\leq i\leq n-m-2$ and $f_{-1}=1$.

\item  \label{month4} Let $\sigma $ be an $m$-simplex in  $K
(\mathbb{F}_p^n)$. The $f$-vector of $\Lk_{K
(\mathbb{F}_p^n)}(\sigma)$ is given by
\begin{equation*}
f_i(\Lk_{K (\mathbb{F}_p^n)}(\sigma)) =
\frac{(p^n-p^{m+1})\cdots (p^n-p^{m+i+1})}{(p-1)^{i+1}(i+1)!}
\end{equation*}
for $0\leq i\leq n-m-2$ and $f_{-1}=1$.
\end{enumerate}
\qed
\end{lemma}

\subsection{The homotopy type of $X(\mathbb F^n_p)$ and $K(\mathbb F^n_p)$} Being finite matroids, the simplicial complexes $X(\mathbb{F}^n_p)$ and $K(\mathbb{F}^n_p)$ have the homotopy type of wedge of spheres. We use~\Cref{month} to calculate the exact number of spheres in those wedges.
\begin{theorem}
\label{mainKp}
\begin{enumerate}
\item [(a)]
Let $n\geq 1$ be an integer. The simplicial complex $X (\mathbb{F}_p^n)$ is homotopy equivalent to the wedge of $A_n(p)$ spheres
$S^{n-1}$, where
\[
A_n(p) =(-1)^n + \sum_{i=0}^{n-1}(-1)^{n-1-i}\frac{(p^n-p^{i})\cdots (p^n - p^0)}{(i+1)!}.
\]

\item [(b)]  Let $n > 1$ be an integer. The simplicial complex $K(\mathbb{F}^n_p)$ is homotopy equivalent to the wedge of $B_n(p)$ spheres $S^{n-1}$,
where
\[
B_n(p) = (-1)^n+\sum_{i=0}^{n-1} (-1)^{n-1-i}\frac{(p^n-p^{i}) \cdots (p^n-p^0)}{(p-1)^{i+1} \cdot (i+1)!}.
\]
\end{enumerate}
\end{theorem}

\begin{proof} 
Since $X(\mathbb{F}^n_p)$ and $K(\mathbb{F}^n_p)$  are finite matroids, they are homotopy equivalent to a wedge of $(n-1)$-spheres. We determine the number of spheres in the wedge using the Euler characteristic~\eqref{euler}, 
\[
\chi(K (\mathbb{F}_p^n))= 1+(-1)^{n-1}B_n(p) = f_0(K (\mathbb{F}_p^n))- \ldots +(-1)^{n-1}f_{n-1}(K (\mathbb{F}_p^n))
\]
which implies that
\begin{equation}
\label{euler2}
B_n(p) = (-1)^n+\sum_{i=0}^{n-1}(-1)^{n-1-i}f_i(K (\mathbb{F}_p^n)).
\end{equation}
Finally, using Lemma~\ref{month2} we obtain
\[
B_n(p) = (-1)^n+\sum_{i=0}^{n-1} (-1)^{n-1-i}\frac{(p^n-p^{i}) \cdots (p^n-p^0)}{(p-1)^{i+1} \cdot (i+1)!}.
\]
The value of $A_n (p)$ follows analogously  from $\chi(X (\mathbb{F}_p^n))= 1+(-1)^{n-1}A_n(p)$.
\end{proof}

\subsection{The homotopy type of $\link_{X(\mathbb F^n_p)}(\sigma)$ and $\link_{K(\mathbb F^n_p)}(\sigma)$}

The links $\link_{X(\mathbb F^n_p)}(\sigma)$ and $\link_{K(\mathbb F^n_p)}(\sigma)$ also have interesting topological and combinatorial properties. Since the group $GL(n,\mathbb{F}_p)$ acts transitively on $X (\mathbb{F}_p^n)$, for integers $n>2$ and $0\leq m \leq n-2$ and $m$-simplices  $\sigma, \tau\in X (\mathbb{F}_p^n)$ the links $\link_{X(\mathbb F^n_p)}(\sigma)$ and $\link_{X(\mathbb F^n_p)}(\tau)$ are isomorphic. The same is true for the links of simplices in $K (\mathbb{F}_p^n)$. 

We prove the following property.
 
 \begin{lemma} 
 \label{link:sh} 
Let $K$ be a countable shellable simplicial complex. Then the link of a simplex in $K$ is shellable.
 \end{lemma}
 
 \begin{proof} For finite complexes this is a classical result 
 (see~\cite{MR1333388} and ~\cite{MR1401765}) which is proved by considering the induced order of facets in the shelling of the simplicial complex.
 
To prove the statement for countably infinite simplicial complexes, we follow the same idea. Let $F_1$, $F_2$, $\dots$, $F_k$, $\dots$ be the ordering of facets in the shelling of an $n$-dimensional simplicial complex $K$ and let $\sigma\in K$. Since for every $k$ the complex $(\cup_{i=1}^{k-1}\overline{F}_i)\cap \overline{F}_k$ is pure and $(n-1)$-dimensional, it is the union of some facets of $F_k$.  Each such facet is the intersection of some $F_i$ with~$F_k$. 
 
 Let $F_{i_1}$, $F_{i_2}$, $\dots$, $F_{i_k}$, $\dots$ be facets in the shelling containing all vertices of $\sigma$ and ordered respecting the shelling order. We prove that $F_{i_1}\setminus \sigma$, $F_{i_2}\setminus\sigma$, $\dots$, $F_{i_k}\setminus\sigma$ is a shelling of $\Lk_K (\sigma)$. Every facet of $\Lk_K (\sigma)$ is of type $F_{i}\setminus \sigma$. Consider $(\cup_{j=1}^{k-1}\overline{F}_{i_j})\cap \overline{F}_{i_k}$ which is the union of some facets of $F_{i_k}$. Let $F_{i_{j_1}}$, $\dots$, $F_{i_{j_s}}$,  $i_{j_1}$, $\dots$, $i_{j_s}< i_k$ be facets that contain all vertices of $\sigma$ such that $F_{i_{j_l}}\cap F_{i_k}$ is a facet of $F_{i_k}$ for $1\leq l\leq s$. Thus,   $(\cup_{j=1}^{k-1}\overline{F}_{i_j})\cap \overline{F}_{i_k}=\cup_{l=1}^{s}(\overline{F}_{i_{j_l}}\cup F_{i_k})$ and therefore $(\cup_{j=1}^{i_{k}-1}\overline{F}_{i_j}\setminus\sigma)\cap \overline{F}_{i_k}\setminus\sigma$ is a pure simplicial complex. This proves the statement.
 \end{proof}

The main result regarding the links of simplices in $X (\mathbb{F}_p^n)$ and $K(\F_p^n)$ is the following.

\begin{theorem}
\label{main2}
\begin{itemize}
\item [(a)]
Let $n,m$ be integers such that $n>1$ and $0\leq m \leq n-2$. The link of an $m$-simplex in $X (\mathbb{F}_p^n)$ is homotopy equivalent to the wedge of $A_{m,n}(p)$ spheres $S^{n-m-2}$, where
\[
A_{m,n}(p)=(-1)^{n-m-1} + \sum_{i=0}^{n-m-2}(-1)^{n-m-i}\frac{(p^n-p^{m+1})\cdots (p^n-p^{m+i+1})}{(i+1)!}.
\]

\item[(b)] Let $n$ and $m$ be integers such that $n>1$ and $0\leq m \leq n-2$. The link of an $m$-simplex $\sigma$ in $K(\mathbb{F}_p^n)$ is homotopy equivalent to the wedge of $B_{m,n}(p)$ spheres $S^{n-m-2}$, where
\[
B_{m,n}(p)=(-1)^{n-m-1} + \sum_{i=0}^{n-m-2}(-1)^{n-m-i}\frac{(p^n-p^{m+1})\cdots (p^n-p^{m+i+1})}{(p-1)^{i+1}(i+1)!}.
\]
\end{itemize}
\end{theorem}

\begin{proof}
By Lemma~\ref{link:sh}, the links $\link_{X(\mathbb F^n_p)}(\sigma)$ and $\link_{K(\mathbb F^n_p)}(\sigma)$ are shellable complexes and therefore have the homotopy type of a wedge of spheres.

Similarly as in Theorem~\ref{mainKp}, using equality~\eqref{euler} and Lemma~\ref{month4} we calculate the number $B_{m,n}(p)$ of spheres in the wedge for $\link_{K(\mathbb F^n_p)}(\sigma)$.

Analogously, the number $A_{m,n}(p)$ of $(n-m-2)$-spheres in the wedge for simplicial complex $\link_{X(\mathbb F^n_p)}(\sigma)$ can be calculated using the Euler characteristic~\eqref{euler} and Lemma~\ref{month3}.
\end{proof}

In commutative algebra Cohen-Macaulay rings and modules have been one of the most studied objects in the last three decades.
A \textit{Cohen-Macaulay ring} is a Noetherian commutative ring $R$ with unit in which any proper ideal $I$ of height $n$ contains a sequence $x_1,\ldots, x_n$ of elements, called a regular sequence, such that for all $i=1,\ldots, n$, the residue class of $x_i$ in the quotient ring $R/\langle x_1,...,x_{i-1}\rangle$ is a non-zero divisor.

Given an abstract simplicial complex $K$ on the vertex set $\{x_1,\ldots,x_m\}$ and a commutative ring~$\k$ with unit, the corresponding \textit{Stanley-Reisner ring} or the face ring of $K$, denoted by $\k[K]$, is the polynomial ring $\k[x_1,\ldots,x_n]$ quotiened by the ideal $I_K$ generated by the square-free monomials corresponding to the non-simplices of $K$,
\[
I_K=\langle x_{i_1}\ldots x_{i_l} |\ \{x_{i_1},\ldots, x_{i_l}\}\notin K\rangle.
\]
A simplicial complex is said to be a \textit{Cohen-Macaulay complex} if its Stanley-Reisner ring is a Cohen-Macaulay ring.

\begin{corollary}
\label{CM}
The universal complexes  $X(\F_p^n)$ and $K(\F_p^n)$ are Cohen-Macaulay.
\end{corollary}
\begin{proof}
By the Reisner criterion~\cite{MR0407036}, a finite simplicial complex $K$ is Cohen-Macaulay over $\k$ if and only if for all simplices $\sigma\in K$, all reduced simplicial homology groups of the link of $\sigma$ in $K$ with coefficients in $\k$ are zero except the top dimensional one. By Theorem~\ref{main2}, the links of an $m$-simplex in $X(\F_p^n)$ and $K(\F_p^n)$ are $(n-m-3)$-connected and $(n-m-2)$-dimensional $CW$-complexes, so the statement follows.
\end{proof}

\section{Proof of Theorem \ref{main:z}}

\subsection{The homotopy type of $X(\Z^n)$}
 We start by listing basic notation needed to prove~\Cref{main:z}.

 Let $e_i$ be the $i$th standard unit vector in $\Z^n$. For every vector $u$, denote by $u(i)$ the $i$th coordinate of $u$.
 
To make technical arguments easier and more straightforward to follow, we assume a lexicographical order of vectors in $\Z^n$, where coordinates are ordered such that $0\prec-1\prec1\prec-2\prec2\prec-3\prec\ldots$.

Recall that by definition $X(\Z^n)$ is a simplicial complex in which a subset $\sigma=\{u_0,u_1,\ldots, u_k\}\subset \Z^n$ is a $k$-simplex  whenever it can be extended to a basis in $\Z^n$. We denote that simplex by $\sigma=u_0u_1\ldots u_k$.
 To vectors $u_0,\ldots, u_{k}$ in $\Z^n$ we associate the $n\times (k+1)$ matrix $U=\Big( u_0,\ldots ,u_k\Big)$ whose columns contain given vectors $u_i, 0\leq i\leq k$. The following lemma gives a criterion for the set $\{u_0,\ldots, u_k\}$ to be a $k$-simplex in $X(\Z^n)$. 
\begin{lemma}
\label{GCD}
A $\sigma=u_0\ldots u_k$ is a $k$-simplex in $X(\Z^n)$ if and only if the greatest common divisor of all $(k+1)\times (k+1)$ minors of the matrix $U$ is equal to one.
\end{lemma}
\begin{proof} Since the vectors $u_0, u_1, \ldots, u_{k}$ belong to a basis of $\Z^n$, there exist vectors $v_{k+1}, \dots, v_{n-1}$ of $\Z^n$ such that $\det \Big(u_0, \ldots, u_k, v_{k+1}, \dots, v_{n-1}\Big)=\pm 1$. For $J=\{j_0, j_1, \dots, j_k\}\subset\{0, 1, \dots, n-1\}$, let $U_J$  denote the $(k+1)\times (k+1)$ minor formed out of  $j_0$-th, $j_1$-st, $\dots$, $j_{k}$-th rows of the matrix $U=\Big( u_0,\ldots , u_k\Big)$.
Let $\rho$ be a permutation of the set $\{i_1, i_2, \dots, i_{n-k-1}\}=\{0, \dots, n-1\}\setminus J$. Then for all $J\subset \{0, 1, \dots, n-1\}$, $|J|=k+1$ and permutations $\rho$ of the set $\{0, 1, \dots, n-1\}\setminus J$, there exist $\varepsilon_{J, \rho}=\pm 1$ such that 
\[
\det \Big(u_0, \ldots,  u_k, v_{k+1}, \ldots, v_{n-1}\Big)=\]
\[
\sum_{\substack{J \subset \{0, 1, \dots, n-1\}\\|J|=k+1} }\sum_{\rho}\varepsilon_{J, \rho} v_{k+1} (\rho (i_1)) v_{k+2} (\rho (i_2))\cdots v_{n-k}(\rho (i_{n-k}))  U_J =\pm 1.
\]
The claim follows directly from the determinant expansion above,
as the greatest common divisor of all $(k+1)\times (k+1)$ minors of the matrix $U$ must divide $1$.

Now we prove that if the greatest common divisor of all $(k+1)\times (k+1)$ is $1$, the set of a vectors $u_0,\ldots, u_{k}$ can be  extended to a basis of $\Z^n$, that is, there is a matrix $ \Big(u_0, \ldots, u_k, v_{k+1}, \ldots, v_{n-1}\Big)$ whose determinant is $\pm 1$. Recall that due to the condition on $(k+1)\times (k+1)$ minors, $\gcd (u_i(1), \dots, u_i(n))=1$ for all $0\leq i\leq k$.

Let $A$ be an integer matrix. Let $T_{(i, j;\lambda)}$, $i\neq j$ be a transformation that adds the $j$th row multiplied by $\lambda$ to the $i$th row and let $S_{(i, j;\lambda)}$ be a transformation that adds $j$th column multiplied by $\lambda $ to the $i$th column. Denote by $\bar{T}_{(i, j;\lambda)}\colon =T_{(i, j;-\lambda)}$  and $\bar{S}_{(i, j;\lambda)}\colon =S_{(i, j;-\lambda)}$. Define the relation $\sim$ on the set of matrices by saying that  $A\sim B$ if matrix $B$ can be obtained from $A$ by applying  transformations $T$ and $S$ finitely many times. The following identities hold
\[
(\bar{T}\circ T)(A)=A \text{ and } (\bar{S}\circ S)(A)=A.
\] 
Thus, the relation $\sim$ is an equivalence relation.

\begin{lemma}\label{apply} Suppose that $A$ is a $n\times k$ matrix and let $\gcd (a_{1, 1}, \dots, a_{1, n})=d.$ Then $A\sim B$, where $B$ is a $n\times k$ matrix such that $b_{1, 1}=d$, $b_{1, 2}=\cdots=b_{1, k}=0$. Moreover, if $d=1$ matrix $B$ may also satisfy that $b_{2, 1}=\cdots=b_{n, 1}=0$.
\end{lemma}

\begin{proof}
For $n=2$, the statement follows by the Euclidian algorithm for $\gcd(a_{1, 1}, a_{1, 2})=d$ which determines a sequence of $T$ transformations, which for $d=1$ are followed by a sequence of $S$ transformations. As $\gcd (a_{1, 1}, \dots, a_{1, n})=\gcd (a_{1, 1}, (a_{1, 2} \dots, a_{1, n}))$, the claim for general $n$ is obtained using induction on $n$ and the base case $n=2$.
\end{proof}

After applying $T_{(i, j;\lambda)}$ to $U$, the $(k+1)\times(k+1)$ minor $U'_J$ corresponding to a set $J=\{j_0, j_1, \dots, j_k\}\subset\{0, 1, \dots, n-1\}$ in matrix $T_{(i, j;\lambda)}(U)$ is equal to $U_J$ if $i\not\in J$ or both $i, j\in J$. If $i\in J$ and $j\not\in J$, then $U'_J=U_J \pm \lambda U_{(J\setminus \{i\})\cup \{j\}}$. However, the greatest common divisor of all $(k+1)\times (k+1)$ minors remains $1$ in $T_{(i, j;\lambda)}(U)$.  The same holds for the matrix $S_{(i, j;\lambda)}(U)$ because this matrix preserves all $(k+1)\times (k+1)$ minors.

Since $\gcd (u_1(1), \dots, u_1(n))=1$, by \Cref{apply} we have that $U\sim \Big(e_1, u'_1, \ldots, u'_k\Big)$ and the greatest common divisor of all $k\times k$ minors of $\Big( u'_1, \ldots, u'_k\Big)$ is equal to $1$. Then $\gcd (u'_1(1), \dots, u'_1(n))=1$ and by applying~\Cref{apply} to $\Big( u'_1, \ldots, u'_k\Big)$  we obtain that $U\sim\Big( e_1, e_2, u''_2, \dots, u''_k\Big) $. Continuing this process in the end, we obtain that $U\sim \Big(e_1, e_2, \dots, e_{k+1}\Big)$. Thus, there is a sequence $P_1, \dots, P_t$, where each $P_i$ is some $T$ or $S$ transformation such that 
\[
P_t \circ (P_{t-1} \circ (\cdots \circ (P_1))\cdots) (U)=\Big(e_1, e_2, \dots, e_{k+1}\Big).
\] 
Then 
\[
\bar{P}_1 \circ (\bar{P}_{2} \circ (\cdots \circ (\bar{P}_t))\cdots) \Big(e_1, e_2, \dots, e_{n-1}\Big)=\Big(u_0, u_1, \dots, u_{k}, v_{k+1}, \dots, v_{n-1}\Big)
\] 
for some vectors $v_{k+1}, \dots, v_{n-1}\in \Z^n$. Our claim follows since the transformations $T$ and $S$ preserve determinant(up to the sign).
\end{proof}

 By abusing notation, each generator of $C_k(X(\Z^n);\Z)$ corresponding to a $k$-simplex $\sigma$ will be also denoted by $\sigma$. Hence a $k$-chain $c$ in $C_k(X(\Z^n);\Z)$ is written as a sum $c=\sum_i\lambda_i\sigma_i=\sum_{i} \lambda_i u_0^iu_1^i\ldots u_k^i$, where $\lambda_i\in\Z$ and $u_0^i\succ\ldots\succ u_k^i$.
 For every permutation $\rho\in\Sigma_k$, we have $u_0u_1\ldots u_k=(-1)^{\sign(\rho)}u_{\rho(0)}\ldots u_{\rho(k)}$.

If $c=\sum_{i} \lambda_i u_0^iu_1^i\ldots u_k^i$ is a chain and $u$ is a vector such that $u=u_j^i$ for some $i$ and $j$, we write $u\in c$.
If $c=\sum_i\lambda_i\sigma_i$ is a chain such that a vector $u$ is a vertex of each simplex $\sigma_i$, 
then we write $c=ub$, where $b$ is the $(k-1)$-chain $b=\sum_i\lambda_i(\sigma_i\setminus \{u\})$.

\begin{proposition}
\label{connected}
For every $n\ge 2$, the simplicial complex $X(\Z^n)$ is connected.
\end{proposition}

\begin{proof}
To show the statement, we show that there is a path between the vertex $e_1$ and any other vertex $v\in\Z^n$  in $X(\Z^n)$. If $v(1)=0$, then $e_1v$ is a 1-simplex in $X(\Z^n)$ by Lemma~\ref{GCD} as $\gcd(v(2),\ldots, v(n))=1$ and thus $e_1$ and $v$ are connected. 

Assume that $v(1)\neq 0$. Then $v$ might not be connected to $e_1$ by a 1-simplex. Instead we show that there is a path $\alpha: v-u_1-\ldots -u_k-e_1$ for some $k$, where $u_i$ are vertices of $X(\Z^n)$ such that $|v(1)|> u_1(1)> \ldots > u_k(1)=0$ and any two consecutive vertices are connected by a 1-simplex. 
Since $v$ is a vertex of $X(\Z^n)$, together with some $n-1$ vectors of $\Z^n$, it forms a basis of $\Z^n$. Choose one of those $n-1$ vectors and denote it by $u_1$. Then by definition, $vu_1$ is a 1-simplex in $X(\Z^n)$. If $u_1(1)<0$ or $u_1(1)\geq |v(1)|$, replace $u_1$ with $u_1+\lambda v$ for some $\lambda\in\Z$
such that $0\le u_1(1)+\lambda v(1) < |v(1)|$. 
Notice that $u_1+\lambda v\in X(\Z^n)$ and that  $v (u_1+\lambda v)$ is a 1-simplex of $X(\Z^n)$. Relabel $u_1+\lambda v$ by $u_1$. We continue constructing the path $\alpha$ until we reach $u_k$ such that $u_k(1)=0$ for some $k$. Then $u_k$ is connected to $e_1$. This can be done in finitely many steps.  Therefore there is a path between $v$ and $e_1$, showing that $X(\Z^n)$ is connected.
\end{proof}

\begin{proposition}
\label{simply connected}
For $n\geq 3$, the simplicial complex $X(\Z^n)$ is simply connected.
\end{proposition}

\begin{proof}
Let $g\in \pi_1(X(\Z^n),e_n)$ and let $\bar\gamma\colon [0,1]\to X(\Z^n)$ be such that $[\bar\gamma]=g$. We consider a piecewise linear approximation of $\bar\gamma$ and use the same notation. Assume that there are $m$ vertices $u_i$ in $\bar\gamma$ which are achieved at  $u_i=\bar\gamma(t_i)$ for $i=0,\ldots ,m$ and that $\bar\gamma (0) = \bar\gamma (1) = e_n$.

Define the integer $M(\bar\gamma):=\max\{u_i(1)\mid i=0,\ldots,m\}$. Choose $\gamma\colon [0,1]\to X(\Z^n)$ such that  $[\gamma]= [\bar\gamma]=g$ which is of the same form as $\bar\gamma$ so that $M(\gamma )$ is as small as possible with respect to $\prec$.

If $M(\gamma)=0$, then by~\Cref{GCD}, a cone $e_1 u_{i-1} u_i\in X(\Z^n)$  for each $i$ and therefore $\Im\gamma\subset \cup_{i=1}^m e_1 u_{i-1} u_i$. Being a cone $\cup_{i=1}^m e_1 u_{i-1} u_i$  is contractible and hence $g=[\gamma]=0$.

If $M(\gamma)\succ 0$, suppose that $\gamma$ is such that $\Omega_{\gamma }(M(\gamma )):=|\{i\mid u_i(1)=M(\gamma)\}| >0$ and as small as possible amongst all representatives of $g$.
Let $k$ be such that $u_k(1)=M(\gamma)$. Note that $k\neq 0,m$ since $u_0(1)=u_m(1)=e_n(1)=0$.

Let $A$ be a matrix with determinant $1$ and with the first two columns $u_{k}$ and $u_{k+1}$. Then $A^{-1}u_{k}=e_{1}$ and $A^{-1}u_{k+1}=e_{2}$. Denote $A^{-1}u_{k-1}$ by $a_{0}$. 

We inductively construct a sequence of vectors  $a_{0},\ldots,a_{l-1}$, such that $\left|a_{0}(2)\right|>a_{1}(2)>\ldots>a_{l-1}(2)=0$.
Assume that a vector $a_{t}$ such that $a_{t}(2)\ne 0$ is defined. There exists a vector $a$ such that $e_1 a_{t}a \in X(\Z^n)$. Then for every vector $a'$ of the form $a'=a+\lambda_0 a_{t}+\lambda_1 e_1$ for $\lambda_0,\lambda_1\in\Z$, the $2$-simplex $e_1 a_{t}a' \in X(\Z^n)$. 

Let  $\lambda_0 \in \Z$ be a uniquely determined number such that $0\le a'(2)< \left|a_{t}(2)\right|$. For such $\lambda_0$, there is a uniquely determined $\lambda_1 \in \Z$ such that
\[
0\le \left(Aa'\right)(1)< \left|u_k(1)\right|.
\]
We define $a_{t+1}$ to be $a'$. 

Since $a_{l-1}e_1 \in X(\Z^n)$ and $a_{l-1}(2)=0$, by Lemma~\ref{GCD} the 3-simplex $e_1a_{l-1}e_2 \in X(\Z^n)$. We denote $e_2$ by $a_l$.
\begin{figure}[h]
\begin{center}
\begin{tikzpicture}[scale=0.8]
\filldraw (0,0)circle (0.05)--node[above]{$\gamma$} (200:2) circle (0.05)
(0,0)--(220:2) circle (0.05)
(0,0)--(240:2) circle (0.05)
(0,0)--(320:2) circle (0.05)
(0,0)--(340:2) circle (0.05);
\draw (200:2)-- (220:2)--(240:2)
(320:2)--(340:2);
\draw[dotted] (240:2) arc (240:320:2) node[midway,above]{$\gamma'$};
\filldraw[gray] (200:2)--++(220:2)node[left]{$u_{k-2}$} circle (0.05)--++(240:0.5)
(340:2)--++(320:2)node[left]{$u_{k+2}$} circle (0.05)--++(300:0.5);

\draw (200:2) node[left]{$u_{k-1}=b_0$} (220:2.4) node{$b_1$} (240:2.4) node{$b_2$}
(320:2.4) node{$b_{l-1}$} (340:2) node[right]{$u_{k+1}=b_l$} (0,0.4) node{$u_k$};
\draw[->] (4.3,-0.5)--node[above]{$A^{-1}$}++(1.5,0);
\draw[->] (5.8,-1.5)--node[above]{$A$}++(-1.5,0);
\begin{scope}[xshift=8.5cm]
\filldraw (0,0)circle (0.05)--(200:2) circle (0.05)
(0,0)--(220:2) circle (0.05)
(0,0)--(240:2) circle (0.05)
(0,0)--(320:2) circle (0.05)
(0,0)--(340:2) circle (0.05);
\draw (200:2)-- (220:2)--(240:2)  (320:2)--(340:2);
\draw[dotted] (240:2) arc (240:320:2);

\draw (200:2.4) node{$a_0$} (220:2.4) node{$a_1$} (240:2.4) node{$a_2$}
(320:2.4) node{$a_{l-1}$} (340:2.4) node{$a_l$} (0,0.4) node{$e_1$};

\end{scope}
\end{tikzpicture}
\end{center}
\caption{}
\end{figure}
Let $b_{i}:=Aa_{i}$ for $i=0,\ldots,l$.
Then $\gamma([t_{k-1}, t_{k+1}])\subset \cup_{i=1}^l u_k b_{i-1}b_i$. So there exists $\gamma'$ such that $[\gamma']=[\gamma]$, $\gamma'(x)=\gamma(x)$ for $x\not\in (t_{k-1}, t_{k+1})$, and $\gamma'([t_{k-1}, t_{k+1} ])\subset \cup_{i=1}^l b_{i-1}b_i$. Hence, $u_k\not\in \Im\gamma'$ and thus
$\Omega_{\gamma'}(M(\gamma ))<\Omega_{\gamma }(M(\gamma ))$ which is a contradiction with $\Omega_{\gamma }(M(\gamma ))$ being minimal. Thus $M(\gamma)=0$ and therefore $g=[\gamma ] = 0 $.
\end{proof}

\begin{proposition}
\label{al}
Let  $n\geq 2$ and $0\le k\le n-2$. Then every reduced homology group $\widetilde{H}_k(X(\Z^n);\Z)$ is trivial.
\end{proposition}

We first prepare technical ground for proving this proposition.
For every $k$-chain $c=\sum_{i=1}^m \lambda_i u_0^iu_1^i\ldots u_k^i$, define $\max(c)=\max\{u_0^i \, |\, 1\le i\le m\}=\max\{u\,|\,u\in c\}$, $M(c)=\max(c)(1)$, and for every $t\in\Z$, let $\Omega_c(t)=|\{u\in c \, |\, u(1)=t\}|$.

\begin{lemma}
\label{granica}
Let $c$ be a $k$-cycle in $C_k(X(\Z^n);\Z)$, $1\leq k \leq n-2$, such that $M(c)=0$. Then $c$ is a boundary.
\end{lemma}

\begin{proof}
Notice that $M(c)=0$ is equivalent to  $\Omega_c(t)=0$ for all $t\ne  0$. By Lemma~\ref{GCD}, since $c$ is a $k$-chain, $e_1c$ is a $(k+1)$-chain. By direct calculation, $\partial(e_1 c)=c$ showing that $c$ is a boundary.
\end{proof}

The order on vertices of $X(\Z^n)$ is such that some vertices have infinitely many smaller vertices. More precisely, we have the following lemma.

\begin{lemma} 
\label{infsmall}
Let $u \in X(\Z^n)$ be a vertex. Then there are infinitely many vertices in $X(\Z^n)$ smaller than $u$ if and only if $u \succeq e_{n-1}$.
\end{lemma}
\begin{proof}
All vectors of the form $(0,\ldots ,0,-1,x)$, $x\in \Z$ are vertices from $X(\Z^n)$ and are smaller than $e_{n-1}$. This proves the sufficient condition.

If $u\prec e_{n-1}$, then first $n-2$ coordinates of $u$ are 0 and the $(n-1)$st coordinate is equal to -1 or 0. If it is -1 then only vertices of the form $(0,\ldots ,0, -1,x ), x\prec u(n)$ and $(0,\ldots, 0, \pm1)$ are smaller than $u$ and there are finitely many of them. If the $(n-1)$st coordinate is 0, then $u=(0,\ldots, 0, \pm1)$ and it could only be bigger than the smallest vertex $(0,\ldots, 0,-1)$.
\end{proof}

Given Lemma~\ref{infsmall}, 
it is not enough to show that every $k$-cycle $c$ such that $M(c)\succ 0$ 
is equivalent to a $k$-cycle $c'$ such that $\max(c')\prec \max(c)$, since finitely many steps may not be enough to get from $c$ to a homologous cycle $\hat c$ such that $M(\hat c)=0$. 
To rectify that, we make use of  $\Omega_c(M(c))$, the number of all vertices of $c$ with first coordinate equal to $M(c)$.

We also need the following technical lemma which we will prove later.

\begin{lemma}
\label{mcvece0}
Let $c\in C_k(X(\Z^n);\Z)$ be a cycle such that $|M(c)|> 0$, $1\leq k \leq n-2$. Then there is a $k$-chain $c'$ homologous to $c$ such that $\max (c')\preceq \max (c)$ and $\Omega_{c'}(M(c))< \Omega_{c}(M(c))$.
\end{lemma}

Now we prove \Cref{al}.

\begin{proof}[Proof of Proposition~\ref{al}]
(i) For $k=0$, the result follows from Lemma~\ref{connected}.

(ii) For $k>0$,  we prove the statement by induction on $n$.\\
Let $n=2$. Then the statement holds by Lemma~\ref{connected}.\\
For $n\ge 3$, let $c=\sum_{i=1}^m\lambda_i \sigma_i=\sum_{i=1}^m\lambda_i u_0^i\ldots u_k^i\in C_k(X(\Z^n);\Z)$ be a $k$-cycle. If $M(c)=0$, then by Lemma~\ref{granica}, $c$ is a boundary. Notice that $M(c)=0$ is equivalent to $\Omega_c(t)=0$ for all $t\neq 0$.

If $|M(c)|>0$, then by Lemma~\ref{mcvece0}, the cycle $c$ is homologous to a $k$-cycle $c'$ for which $\max (c')\preceq \max (c)$ and $\Omega_{c'}(M(c))<\Omega_{c}(M(c))$. Iterating this process first we find a cycle $\bar c$ homologous to $c $ with $\max (\bar c)\preceq \max (c)$ and $\Omega_{\bar c}(M(c))=0$, that is, $M(\bar c) \prec M(c)$ and then a $k$-cycle $\hat c$ which is homologous to $c$ and $M(\hat c)=0$. Since $M(\hat c)=0$,  by \Cref{granica} the cycle $\hat c$ is a boundary and so is $c$.
\end{proof}

To prove~\Cref{mcvece0} we explicitly construct a $k$-chain $c'$. We start by outlining the main steps of the construction which are also depicted in
\Cref{sketch}. Let  $c=\sum_i\lambda_i\sigma_i$ be a $k$-cycle and let $u$ be the largest vertex in $c$, that is $u = \max (c)$. Partition the cycle $c$ into simplices that contain or do not contain the vertex $u$. Consider those simplices that contain $u$ and denote by $c_u$ the $(k-1)$-chain such that  $uc_u=\sum_{u\in \sigma_i}\lambda_{i}\sigma_i $. 

We want to replace $uc_u$ by a smaller homologous chain using that $c_u$ is a $(k-1)$-cycle and the induction hypothesis that $\widetilde{H}_k(X(\Z^{n-1});\Z)=0$ for $k\le n-3$.  To simplify the form of our simplices we change the coordinate system by  a matrix $A$ with determinant 1 and with the first column $u$. The matrix $A^{-1}$ induces a simplicial automorphism of $C_*(X(\Z^n);\Z)$; that is, a $k$-simplex $\sigma$ which is represented by $n\times(k+1)$ matrix is mapped to the $n\times(k+1)$ matrix $A^{-1}\sigma$. In that way $A^{-1}(u)=e_1$,
so if $u\in \sigma$ the simplex $A^{-1}\sigma$ has $e_1$ as one of its vertices. Denote by $\pi\colon \Z^n\to\Z^{n-1}$ the projection to the last $n-1$ coordinates. Then, $\pi(A^{-1}(c_u))$ is a cycle in $C_{k-1}(X(\Z^{n-1});\Z)$ and by the induction hypothesis there exists a $k$-chain $d$ such that $\partial d=\pi(A^{-1}(c_u))$. To finish, we need $\iota \colon \Z^{n-1} \lra \Z^n$ which is a section of $\pi $. The map $\iota $ cannot be always defined with all the properties needed. Thus we introduce a new invariant $L$ on chains, so that when $L(c)=0$ we can define $\iota$. If $L(c)>0$ we construct a $k$-chain $c''$ homologous to $c$ with $\max (c'')=\max (c)$, $\Omega_{c''}(M(c))=\Omega_c(M(c))$ and $L(c'')<L(c)$. Therefore, in at most $L(c)$ steps we will obtain $c''$ with $L(c'')=0$ and construct $\iota$ and a chain $c'$ in terms of $\iota$ and $d$ required by~\Cref{mcvece0}.

\begin{center}
\begin{figure}
\begin{tikzpicture}[scale=0.9]
\draw (0:1.3)--node[above right]{$c_u$}(30:1.0)--(65:0.8)--(100:0.7)--(140:0.9)--(170:1.1)--(230:0.7)--(290:0.8)--(340:1.2)--cycle;
\node(0,0) at (-90:1.1) {$\sum_{i\in S} \lambda_i u_1^i\ldots u_k^i$};
\draw[->] (2,0)--node[above]{$A^{-1}$}(3,0);
\begin{scope}[xshift=5cm]
\draw (0:1.3)--node[above right]{$A^{-1}(c_u)$}(30:1.0)--(65:0.8)--(100:0.7)--(140:0.9)--(170:1.1)--(230:0.7)--(290:0.8)--(340:1.2)--cycle;
\node(0,0) at (-90:1.1) {$\sum_{i\in S} \lambda_i v_1^i\cdots v_n^i$};
\end{scope}
\draw[->] (7,0)--node[above]{$\pi$}(8,0);
\begin{scope}[xshift=10cm]
\draw (0:1.3)--node[above right]{$\pi\left(A^{-1}(c_u)\right)$}
(30:1.0)--(65:0.8)--(100:0.7)--(140:0.9)--(170:1.1)--(230:0.7)--(290:0.8)--(340:1.2)--cycle;
\node(0,0) at (-90:1.1) {$\sum_{i\in S} \lambda_i \pi(v_1^i)\cdots \pi(v_n^i)$};
\end{scope}
\draw[->,decorate,decoration={snake,amplitude=.4mm,segment length=2mm,post length=1mm}] (10,-1.5)--node[left,text width=3.5cm]{exists a chain $d$ such that $\partial d=\pi\left(A^{-1}(c_u)\right)$}++(0,-2);
\begin{scope}[yshift=-4.5cm]
\draw[red!50] (-0.55,0.2)--(0.4,0)--(0,0.4)--cycle;
\draw[red!50] (-0.55,0.2)--(-0.1,-0.3)--(0.4,0);
\draw[red!50] (0.4,0)--(0:1.3);
\draw[red!50] (0.4,0)--(30:1.0)--(0,0.4)--(65:0.8);
\draw[red!50] (0,0.4)--(100:0.7);
\draw[red!50] (0,0.4)--(140:0.9);
\draw[red!50] (-0.55,0.2)--(140:0.9);
\draw[red!50] (-0.55,0.2)--(170:1.1);
\draw[red!50] (-0.55,0.2)--(230:0.7)--(-0.1,-0.3)--(290:0.8)--(0.4,0);
\draw[red!50] (290:0.8)--(0:1.3);
\draw[red!50] (0,2)--(-0.55,0.2) (0,2)--(0.4,0) (0,2)--(0,0.4) (0,2)--(-0.1,-0.3);
\draw[gray] (0:1.3)--(30:1.0)--(65:0.8)--(100:0.7)--(140:0.9);
\draw (140:0.9)--(170:1.1)--(230:0.7)--(290:0.8)--(340:1.2)--(0:1.3);
\draw[gray] (0,2)--(30:1.0) (0,2)--(65:0.8) (0,2)--(100:0.7) (0,2)--(140:0.9);
\draw (0,2)--(0:1.3) (0,2)--(170:1.1) (0,2)--(230:0.7) (0,2)--(290:0.8) (0,2)--(340:1.2);
\node[red](0,0) at (-90:1.1) {$u A\left(\iota(d)\right)$};
\node at(0.3,2) {$u$};
\draw[<-] (2,0)--node[above]{$A$}(3,0);
\begin{scope}[xshift=5cm]
\draw[red!50] (-0.55,0.2)--(0.4,0)--(0,0.4)--cycle;
\draw[red!50] (-0.55,0.2)--(-0.1,-0.3)--(0.4,0);
\draw[red!50] (0.4,0)--(0:1.3);
\draw[red!50] (0.4,0)--(30:1.0)--(0,0.4)--(65:0.8);
\draw[red!50] (0,0.4)--(100:0.7);
\draw[red!50] (0,0.4)--(140:0.9);
\draw[red!50] (-0.55,0.2)--(140:0.9);
\draw[red!50] (-0.55,0.2)--(170:1.1);
\draw[red!50] (-0.55,0.2)--(230:0.7)--(-0.1,-0.3)--(290:0.8)--(0.4,0);
\draw[red!50] (290:0.8)--(0:1.3);
\draw[red!50] (0,2)--(-0.55,0.2) (0,2)--(0.4,0) (0,2)--(0,0.4) (0,2)--(-0.1,-0.3); 
\draw[gray] (0:1.3)--(30:1.0)--(65:0.8)--(100:0.7)--(140:0.9);
\draw (140:0.9)--(170:1.1)--(230:0.7)--(290:0.8)--(340:1.2)--(0:1.3);
\draw[gray] (0,2)--(30:1.0) (0,2)--(65:0.8) (0,2)--(100:0.7) (0,2)--(140:0.9);
\draw (0,2)--(0:1.3) (0,2)--(170:1.1) (0,2)--(230:0.7) (0,2)--(290:0.8) (0,2)--(340:1.2);
\node[red](0,0) at (-90:1.1) {$e_1 \iota(d)$};
\node at(0.3,2) {$e_1$};
\end{scope}
\draw[->,decorate,decoration={snake,amplitude=.4mm,segment length=2mm,post length=1mm}] (8,0)--node[above]{$\iota$}(7,0);
\begin{scope}[xshift=10cm]
\draw[red] (-0.55,0.2)--(0.4,0)--(0,0.4)--cycle;
\draw[red] (-0.55,0.2)--(-0.1,-0.3)--(0.4,0);
\draw[red] (0.4,0)--(0:1.3);
\draw[red] (0.4,0)--(30:1.0)--(0,0.4)--(65:0.8);
\draw[red] (0,0.4)--(100:0.7);
\draw[red] (0,0.4)--(140:0.9);
\draw[red] (-0.55,0.2)--(140:0.9);
\draw[red] (-0.55,0.2)--(170:1.1);
\draw[red] (-0.55,0.2)--(230:0.7)--(-0.1,-0.3)--(290:0.8)--(0.4,0);
\draw[red] (290:0.8)--(0:1.3);
\draw (0:1.3)--(30:1.0)--(65:0.8)--(100:0.7)--(140:0.9)--(170:1.1)--(230:0.7)--(290:0.8)--(340:1.2)--cycle;
\node[red](0,0) at (-90:1.1) {$d$};
\end{scope}
\end{scope}
\end{tikzpicture}
\caption{A schematic diagram of the construction of an intermediate cycle $c'$. On the left hand side, chains are in $C_*(X(\Z^n);\Z)$, in the middle chains are in $C_*(X(\Z^n);\Z)$ but vectors are written in a new basis which has $u$ as the first coordinate vector, and on the right hand side chains are in $C_*(X(\Z^{n-1});\Z)$.}\label{sketch}
\end{figure}
\end{center}
 
Let $f\colon \Z^n\to \Z^m$ be a map and let  $\sigma=u_0u_1\ldots u_k$ be a $k$-simplex in $X(\Z^n)$. Define $f(\sigma)$ to be $f(u_0)f(u_1)\ldots f(u_k)$. Note that $f(\sigma)$ is not necessary a simplex in $X(\Z^m)$ as $\{f(u_0),f(u_1),\ldots, f(u_k) \}$ might not be extendable to a basis of $\Z^n$. 

\begin{lemma}
\label{matrix}
Let $A\colon \Z^n\to\Z^n$ be a linear map such that $\det A=\pm 1$. Then $A$ induces a simplicial isomorphism $A\colon X(\Z^n) \to X(\Z^n)$. 
\end{lemma}
\begin{proof}
Because $\det A = \pm 1$, there is a linear map $A^{-1}\colon \Z^n \lra \Z^n$ that sends bases into bases. Therefore, maps $A$ and $A^{-1}$ map simplices of $X(\Z^n)$ into same dimensional simplices of $X(\Z^n)$ and thus induce simplicial maps which are inverse to each other.
\end{proof}

\begin{lemma}
\label{e1 pi}
Let $\pi\colon \Z^n\to\Z^{n-1}$ be the projection onto the last $n-1$ coordinates. Then $\pi (u_1) \pi (u_2)\ldots \pi (u_k)\in X(\Z^{n-1})$  if and only if $e_1u_1u_2\ldots u_k \in X(\Z^n)$.
\end{lemma}
 \begin{proof}
As the set of $k\times k$ minors of the $(n-1) \times k$ matrix $\Big ( \pi(u_1) \cdots \pi (u_k) \Big)$ is equal to the set of all $(k+1)\times (k+1)$ minors of the $n \times (k+1)$ matrix $\Big( e_1, u_1, \dots ,u_k \Big)$, the statement follows by~Lemma~\ref{GCD}.
 \end{proof}

\begin{lemma}
\label{iota}
Let $\iota\colon \Z^{n-1}\to\Z^n$ be any map such that $\pi\circ \iota=\Id$. 
Then $e_1\iota (\sigma)\in X(\Z^n)$ for all $\sigma \in X(\Z^{n-1})$ and $\iota$ induces a simplicial map $\iota\colon X(\Z^{n-1})\to X(\Z^n)$.
\end{lemma}
\begin{proof}
Let $\sigma = v_1\ldots v_k \in X(\Z^{n-1})$, $1\leq k\leq n-1$. The greatest common divisor of the set of all $(k+1)\times (k+1)$ minors of the $n \times  (k+1)$ matrix $\Big( e_1, \iota (v_1),\dots ,\iota (v_k) \Big)$ is equal to the greatest common divisor of the set of all $k\times k$ minors of the $(n-1)\times k$ matrix $\Big( v_1,\dots  ,v_k \Big)$ which is equal to 1 and hence $e_1\iota(\sigma )\in X(\Z^n)$. Since $e_1\iota(\sigma )\in X(\Z^n)$, we have $\iota(\sigma )\in X(\Z^n)$, proving that $\iota $ induces a simplicial map.
\end{proof}

Let  $u=\max(c)$ and let $S=\{i \,|\ u_0^i=u\}$. Then $S\ne\emptyset$. Note that $u\ne u_j^i$ for any $i$ and $j\ge 1$. Since $c$ is a cycle, expanding the boundary $\partial c$, we have
\begin{equation}
\label{boundary}
\begin{array}{rl}
0=&\partial \left(\sum_{i\in S}\lambda_i u_0^i\ldots u_k^i+\sum_{i\not\in S}\lambda_i u_0^i\ldots u_k^i\right)\\=&
\sum_{i\in S}\lambda_i u_1^i\ldots u_k^i-u\partial \left(\sum_{i\in S}\lambda_i u_1^i\ldots u_k^i\right)+\partial\left(\sum_{i\not\in S}\lambda_i u_0^i\ldots u_k^i\right).
\end{array}
\end{equation}
Denote by $c_u$ the sum $c_u=\sum_{i\in S}\lambda_i u_1^i\ldots u_k^i$. Since the middle summands in \eqref{boundary} are the only one that contain the vector $u$, $u\partial (c_u)=0$ and thus $c_u$ is a $(k-1)$-cycle in $C_{k-1}(X(\Z^n);\Z)$.

To find a condition when a vertex $v$ is in the link of $u$, we write vertices in a new basis with the first vector $u$. In order to do that we choose a matrix $A=(a_{i,j})_{1\le i,j\le n}$ such that its first column is the vector $u$ and $\det A=1$. Then $Ae_1=u$ and consequently $A^{-1}u=e_1$. For all $i$ and $j$, define $v_j^i=A^{-1}u_j^i$. Then the $(k-1)$-cycle $c_u$ is mapped to a $(k-1)$-cycle $A^{-1}(c_u)=\sum_{i\in S}\lambda_i v_1^i\ldots v_k^i$ in $C_{k-1}(X(\Z^{n});\Z)$.

Let $\pi\colon \Z^n\to \Z^{n-1}$ be the projection on the last $n-1$ coordinates. By \Cref{e1 pi}, $\pi(A^{-1}(c_u))$ is a $(k-1)$-cycle in $C_{k-1}(X(\Z^{n-1});\Z)$. In general there is no inclusion $\iota\colon\Z^{n-1}\to\Z^n$ such that $\iota(\pi(A^{-1}(c_u)))=A^{-1}(c_u)$, since the restriction of $\pi$ on $c_u$ is not necessarily injective. 
If for all $v\in c_u$ we have $0\le v(1)<|u(1)|$ then the following lemma allows us to define $\hat\iota\colon \Z^{n-1}\to \Z^n$ such that $\hat\iota(\pi(A^{-1}(c_u)))=A^{-1}(c_u)$

\begin{lemma}
\label{uniquesol}
For given integers $x_2,\ldots ,x_n$, $a_{1,j}$, $a_{1,1}\neq 0$, there is a unique integer
$x_1$ for which the inequality 
\[
0\le a_{1,1} x_1+a_{1,2}x_2+\ldots +a_{1,n}x_n<|a_{1,1}|
\]
holds.
\end{lemma}
\begin{proof}
There is a unique number $r$ with $0\leq r < |a_{1,1}|$ such that 
$a_{1,2}x_2+\ldots +a_{1,n}x_n \equiv r \pmod{|a_{1,1}|}$. Then
\[
x_1 := -(a_{1,2}x_2+\ldots +a_{1,n}x_n -r )/a_{1,1}. \qedhere
\]
\end{proof}

By Lemma~\ref{uniquesol}, if $0\le u_j^i(1)<|u(1)|$ the number $v_j^i(1)$ is uniquely determined by the numbers $v_j^i(2),\ldots, v_j^i(n)$ and the matrix $A$ as it is a unique solution of the equation $u_j^i(1)=a_{1,1}v_j^i(1)+\ldots+ a_{1,n}v_j^i(n)$ having in mind that $a_{1,1}=u(1)$.

That allows us to define the following maps.
\begin{definition}
\label{pi}
Define $\pi\colon \Z^n\to \Z^{n-1}$ and $\hat\iota\colon\Z^{n-1}\to\Z^n$ by $\pi(x_1,\ldots,x_n)=(x_2,\ldots,x_n)$ and $\hat\iota(x_2,\ldots,x_n)=(x_1,\ldots,x_n)$, where by Proposition~\ref{uniquesol} $x_1$ is the only solution of the inequality $0\le a_{1,1}x_1+\ldots +a_{1,n}x_n<|u(1)|$.
\end{definition}

Therefore, we have $\hat\iota(\pi(v_j^i))=v_j^i$ if and only if $0\le u_j^i(1)<|u(1)|$. 
If there are vectors in $c_u$ with the first coordinate negative or equal to $|u(1)|$,  then we need to define an inclusion $\Z^{n-1}\to\Z^n$ which modifies the inclusion $\hat\iota$ in finitely many vectors. The inclusion depends on the number of vertices $v\in c_u$ with $v(1)<0$ or $v(1)= |u(1)|$. Therefore we define
\[
L(c) := \max_{i\in S}
|\{ j \text{ }|\text{ } \hat\iota (\pi (v_j^i))\neq v_j^i \in A^{-1}(c_u) \}| . 
\]

\noindent{\it Proof of \Cref{mcvece0}.}
\subsection*{(i) Let $L(c)=0$}
Then $\hat\iota(\pi(v))=v$ for all $v\in A^{-1}(c_u)$.
Because $e_1A^{-1}(c_u)$ is a $k$-chain in $C_{k}(X(\Z^n);\Z)$, by \Cref{e1 pi} $\pi(A^{-1}(c_u))$ is a $(k-1)$-chain in $C_{k-1}(X(\Z^{n-1});\Z)$. Furthermore, since $c_u$ is a cycle so is $\pi(A^{-1}(c_u))$. By the induction hypothesis, there exists a $k$-chain $d\in C_k(X(\Z^{n-1});\Z)$ such that $\partial(d)=\pi(A^{-1}(c_u))$. Then $A^{-1}(c_u)=\hat\iota(\pi(A^{-1}(c_u)))=\hat\iota(\partial (d))=\partial (\hat\iota(d))$, and hence by Lemma~\ref{iota}
\[
\partial(e_1\hat\iota(d))=\hat\iota(d)-e_1\partial(\hat\iota(d))=\hat\iota(d)-e_1 A^{-1}(c_u)=\hat\iota(d)- A^{-1}(u c_u).
\]
Therefore, the homology class of $u c_u- A(\hat\iota(d))$ is trivial and we constructed a chain $c'=A(\hat\iota(d))+\sum_{i\not\in S} \lambda_i u_0^i\ldots u_k^i$ such that $c$ is homologous to $c'$. Note that, $A(\hat \iota (v)) \prec u$ for all vertices $v\in X(\Z^{n-1})$ by definition of $\hat \iota$, and therefore $\max (A(\hat \iota (d)))\prec u $. Because $\max (A(\hat\iota(d)))\prec \max (u c_u)=\max (c)$, we get $\Omega_{c'}(M(c))<\Omega_{c}(M(c))$.

\subsection*{(ii) Let \texorpdfstring{$L(c)\geq 1$}{$l(c)>1$}}
Recall that $A^{-1}(c_u)=\sum_{i\in S}\lambda_i\tau_i=\sum_{i\in S}\lambda_iv_1^i\ldots v_k^i$. Define
\begin{multline*}
\mathcal{M}=\{w_1\ldots w_L \in X(\Z^n)\,|\, \\
\hat\iota (\pi (w_k))\neq w_k \text{ for all } k \text{ and there is }  i \in S \text{ such that }w_k\in \tau_i \text{ for all }k\}.
\end{multline*}
Let $T_{w_1\ldots w_L}=\{ i\in S\,|\, w_1,\ldots ,w_L\in \tau_i \}$ and $T:=\cup_{w_1\ldots w_L\in \mathcal{M}} T_{w_1\ldots w_L}$. 
Note that $T_{w_1\ldots w_L}\cap T_{w'_1\ldots w'_L}=\emptyset$ if $w_1\ldots w_L\neq w'_1\ldots w'_L$ by the definition of $L$. 

Denoting 
\[
\sum_{i\in T_{w_1\ldots w_L}}\lambda_iv_1^i\ldots v_k^i = w_1\ldots w_L b_{w_1\ldots w_L}
\]
for a chain $b_{w_1\ldots w_L}\in C_{k-L-1}(X(\Z^n);\Z)$, rewrite 
\[
A^{-1}(c_u) = w_1\ldots w_L b_{w_1\ldots w_L} +\sum_{i\notin T_{w_1\ldots w_L}}\lambda_iv_1^i\ldots v_k^i
\]
and
\[
0=\partial (A^{-1}(c_u)) = \partial (w_1\ldots w_L)b_{w_1\ldots w_L} + (-1)^L w_1\ldots w_L \partial (b_{w_1\ldots w_L}) + 
\]
\[
+ \partial \Big( \sum_{i\notin T_{w_1\ldots w_L}}\lambda_iv_1^i\ldots v_k^i \Big).
\]
Because the second summand is the only one containing $w_1\ldots w_L$, then   $b_{w_1\ldots w_L}$ is a cycle for all $w_1\ldots w_L \in \mathcal{M}$.

\begin{definition}
\label{mu1}
Let $\nu \colon \Z^n \lra \Z^n$ be a map given by
\[
\nu (x_1 , \ldots , x_n) = (-x_1, \ldots , -x_n).
\]
\end{definition}
Note that the map $\nu$ extends to a nondegenerative simplicial map $\nu \colon X(\Z^n) \lra X(\Z^n)$ with the following property: for any vertex $w$ and a simplex $\sigma$ in $X(\Z^n)$, we have $w\sigma \in X(\Z^n) $ if and only if $\nu(w)\sigma \in X(\Z^n)$.

\begin{definition}
\label{mu L}
For $k\geq 0$ and $\mu = (\mu_0 , \ldots , \mu_k)\in \{-1,\, 1\}^{k+1}$, define a map $\mu\colon C_{k}(X(\Z^n); \Z)\to C_{k}(X(\Z^n); \Z)$ by 
$\mu(u_0\ldots u_k)= \mu_0(u_0)\ldots \mu_k(u_k)$, where 
\[
\mu_i(u_i)=\left\{\begin{array}{ll}
u_i & \text{ if }\mu_i=1\\
\nu (u_i) & \text{ if } \mu_i=-1.\\
\end{array}\right.
\]
Let $|\mu | = |\{i\,|\, \mu_i = -1 \}|$.
\end{definition}
Note that for a $k$-simplex $\sigma \in X(\Z^n)$, $\mu (\sigma )$ is a $k$-simplex . We have the following property of the map $\mu$.

\begin{lemma}
\label{n-1 cycle}
For $k\geq 0$ and every simplex $u_0\ldots u_k \in X(\Z^n)$,
\begin{equation}
    \label{formula}
\sum_{\mu \in \{-1,\, 1\}^{k+1}} (-1)^{|\mu |}\partial (\mu (u_0\ldots u_k)) = 0.
\end{equation}
\end{lemma}
\begin{proof}
We prove the lemma by induction on $k$.

For $k=0$, we have
\[
\partial u_0 - \partial \nu (u_0) = 0.
\]
Now assume that the lemma is true for $k$. Then the left hand side of \eqref{formula} is equal to
\[
\sum_{\mu \in \{-1,\, 1\}^{k}} (-1)^{|\mu |}\partial (\mu (u_0\dots u_{k-1})u_k) - \sum_{\mu \in \{-1,\, 1\}^{k}} (-1)^{|\mu  |}\partial (\mu (u_0\ldots u_{k-1})  \nu (u_k))=
\]
\[
=\sum_{\mu \in \{-1,\, 1\}^{k}} (-1)^{|\mu |}\partial (\mu (u_0\ldots u_{k-1})) u_k + (-1)^{k}\sum_{\mu \in \{-1,\, 1\}^{k}} (-1)^{|\mu |}\mu (u_0\ldots u_{k-1}) -
\]
\[
-\sum_{\mu \in \{-1,\, 1\}^{k}} (-1)^{|\mu |}\partial (\mu (u_0\ldots u_{k})) \nu (u_k) - (-1)^{k}\sum_{\mu \in \{-1,\, 1\}^{k}} (-1)^{|\mu |}\mu (u_0\ldots u_{k-1})=0
\]
because by induction hypothesis, the first and the third sums are 0, while the second and the fourth sums cancel each other.
\end{proof}

Because $e_1w_1\ldots w_L b_{w_1\ldots w_L}$ is a $k$-chain, $e_1\mu (w_1\ldots w_L) b_{w_1\ldots w_L}$ is a $k$-chain for every $\mu$. Furthermore,
\[
\partial \Big( \sum_{\mu \in \{-1,\, 1\}^L} (-1)^{|\mu |}\mu (w_1\ldots w_L)b_{w_1\ldots w_L}\Big) =
\]
\[
=\sum_{\mu \in \{-1,\, 1\}^L} (-1)^{|\mu |}\partial (\mu (w_1\ldots w_L))b_{w_1\ldots w_L}- \sum_{\mu \in \{-1,\, 1\}^L} (-1)^{|\mu |+L}\mu (w_1\ldots w_L)\partial(b_{w_1\ldots w_L})
\]
\[
=0.
\]
By induction hypothesis, there is a $k$-chain $d_{w_1\ldots w_L}\in C_k(X(\Z^{n-1});\Z)$ such that
\[
\partial (d_{w_1\ldots w_L}) = \pi \Big(\sum_{\mu \in \{-1,\, 1\}^L} (-1)^{|\mu |}\mu (w_1\ldots w_L)b_{w_1\ldots w_L}\Big).
\]
For $w_1\ldots w_L \in \mathcal{M}$, define a map $\iota_{w_1\ldots w_L}\colon \Z^{n-1} \lra \Z^n$ by
\[
\iota_{w_1\ldots w_L}(w) =\left\{\begin{array}{ll} \hat\iota (w)& \text{ if } w\neq \pi (w_i) \text{ for all } i \\
w_i & \text{ if } w=\pi (w_i).\\
\end{array}\right.
\]
Note that $\hat\iota (\pi (v)) =v $ for all vertices of $b_{w_1\ldots w_L}$ because maximal possible number of  vertices from simplices in $A^{-1}(c_u)$ with $\hat\iota (\pi (v)) \neq v $ is $L$, $w_1\ldots w_Lb_{w_1\ldots w_L}$ is made out of simplices from $A^{-1}(c_u)$ and $\hat\iota (\pi (w_i)) \neq w_i $. Therefore, $\hat\iota (\pi (b_{w_1\ldots w_L}))= b_{w_1\ldots w_L}$. 
Hence
\[
\partial (\iota_{w_1\ldots w_L}(d_{w_1\ldots w_L})) = w_1\ldots w_L b_{w_1\ldots w_L} - \sum_{\genfrac{}{}{0pt}{}{\mu\in\{-1,1\}^L}{\mu\ne(1,\ldots,1)}} (-1)^{|\mu |} \iota_{w_1\ldots w_L}(\pi(\mu(w_1\ldots w_L)))b_{w_1\ldots w_L}
\]
and
\[
\partial (e_1\iota_{w_1\ldots w_L}(d_{w_1\ldots w_L})) = \iota_{w_1\ldots w_L}(d_{w_1\ldots w_L}) - e_1 \partial (\iota_{w_1\ldots w_L}(d_{w_1\ldots w_L})).
\]
Therefore
\[
e_1w_1\ldots w_L b_{w_1\ldots w_L} \]
is homologous to 
\[\iota_{w_1\ldots w_L}(d_{w_1\ldots w_L}) + e_1\sum_{\genfrac{}{}{0pt}{}{\mu\in\{-1,1\}^L}{\mu\ne(1,\ldots,1)}} (-1)^{|\mu |} \iota_{w_1\ldots w_L}(\pi(\mu(w_1\ldots w_L)))b_{w_1\ldots w_L}.
\]
Finally,
\[
c = A\Big( e_1A^{-1}(c_u) + \sum_{i\notin S}\lambda_iv_0^i\ldots v_k^i \Big) =
\]
\[
=A\Big( \sum_{w_1\ldots w_L\in\mathcal{M}}e_1w_1\ldots w_L b_{w_1\ldots w_L} + \sum_{i\in S\setminus T} \lambda_iv^i +\sum_{i\notin S}\lambda_iv^i\Big)
\]
is homologous to
\[
c'' := A \Big(\sum_{w_1\ldots w_L\in\mathcal{M}}  \iota_{w_1\ldots w_L}(d_{w_1\ldots w_L})\Big) + 
\]
\[
+ u A\Big(\sum_{w_1\ldots w_L\in\mathcal{M}}\sum_{\genfrac{}{}{0pt}{}{\mu\in\{-1,1\}^L}
{\mu\ne(1,\ldots,1)}} (-1)^{|\mu |} \iota_{w_1\ldots w_L}(\pi(\mu(w_1\ldots w_L)))b_{w_1\ldots w_L} \Big) +\sum_{i\notin T }\lambda_iu^i.
\]
Note that $\max (c'') \preceq u = \max (c)$ and therefore $M(c'') \preceq M(c)$. Moreover, $\Omega_{c''}(M(c))\leq \Omega_c(M(c))$ 
because we did not add a new vertex whose first coordinate is equal to $|u(1)|$ or larger using $\iota$'s. If $\max (c'')\prec \max (c)$ or $\Omega_{c''}(M(c)) < \Omega_c(M(c))$ then we are done since $c''$ is required cycle $c'$. 

If $\max (c'') = \max (c)$ and $\Omega_{c''}(M(c)) = \Omega_c(M(c))$ then $L(c'')<L(c)$ as in each summand in
\[
u A\Big(\sum_{w_1\ldots w_L\in\mathcal{M}}\sum_{\genfrac{}{}{0pt}{}{\mu\in\{-1,1\}^L} {\mu\ne(1,\ldots,1)}} (-1)^{|\mu |} \iota_{w_1\ldots w_L}(\pi(\mu(w_1\ldots w_L)))b_{w_1\ldots w_L} \Big)
\]
there is at least one less vertex $w$ such that $\iota_{w_1\ldots w_L} \pi (w) \neq w$ as at least one coordinate $j$ of $\mu$ is equal to $-1$ implying that $\mu (w_1\ldots w_L )$ contains $\nu (w_j)$ and $\iota_{w_1\ldots w_L} \pi (\nu (w_j)) \neq \nu (w_j)$, because $\pi (\nu (w_j)) \neq \pi (w_i)$ for each $i\neq j$.

Therefore we have constructed the required cycle $c''$ and completed the proof of~\Cref{mcvece0}.
\qed

\subsection{Links in \texorpdfstring{$X(\Z^n)$}{X(Zn)}}

Recall that for a simplicial complex $K$, the link of $\sigma$ in $K$ is defined by
\[
\Lk_K(\sigma) = \{\tau \in K\, | \, \tau \cup \sigma \in K ,\, \tau \cap \sigma = \emptyset \}.
\]

\begin{lemma}
\label{linkiso}
Let $\sigma_1 , \sigma_2 \in X(\Z^n)$ be two $m$-simplices, $0\leq m \leq n-1$. Then there is a simplicial homeomorphism $A\colon X(\Z^n)\lra X(\Z^n)$ which maps $\sigma_1$ into $\sigma_2$ and maps $\Lk_{X(\Z^n)}(\sigma_1 )$ isomorphicly onto $\Lk_{X(\Z^n)} (\sigma_2 )$.
\end{lemma}
\begin{proof}
It is enough to prove the statement for $\sigma_2 = e_1\ldots e_{m+1}$.
Let $\sigma_1 =  u_1\ldots u_{m+1} \in X(\Z^n)$. Then $\sigma$ can be extended by some vectors $u_{m+2}, \ldots ,u_n \in \Z^n$ to a basis $u_1,\ldots ,u_n$ of $\Z^n$. Then the matrix $A := \Big( u_1, \dots ,u_n \Big)$ has determinant $\pm 1$ and $Ae_i = u_i$ for all $i$. 

By Lemma~\ref{matrix}, the map $A^{-1} \colon X(\Z^n)\lra X(\Z^n)$ is a simplicial homeomorphism which maps $\sigma_1$ onto $\sigma_2$ and moreover $A^{-1}$ maps isomorphicly $\Lk(\sigma_1 )$ onto $\Lk(\sigma_2)$.
\end{proof}

By \Cref{linkiso}, all links of $(m-1)$-simplices in $X(\Z^n)$ are simplicially isomorphic. We consider the $(m-1)$-simplex 
\begin{equation}
\label{sigma}
\sigma := e_{n-m+1}\ldots e_n , 1\leq m \leq n-1. 
\end{equation}

\begin{lemma}
\label{pi-link}
\begin{itemize}

\item[(a)] Let $\pi ' \colon \Z^n \lra \Z^{n-m}$ be the projection onto the first $n-m$ coordinates. Then 
    \[
    \tau \in \Lk_{X(\Z^n)} (\sigma ) \text{ if and only if } \pi '(\tau) \in X(\Z^{n-m}). 
    \]

\item[(b)] Let $A$ be an $n\times n$ matrix with determinant 1 with the last $m$ columns  $e_{n-m+1}, \ldots , e_n$. Then $A$ maps $\Lk_{X(\Z^n)}(\sigma )$ isomorphicly onto $\Lk_{X(\Z^n)}(\sigma )$.

\item[(c)] Let $\pi \colon \Z^n \lra \Z^{n-1}$ be the projection onto the last $n-1$ coordinates. Then 
    \[
    \pi \colon \Lk_{X(\Z^n)}(e_1\sigma ) \lra \Lk_{X(\Z^{n-1})}(\pi (\sigma )) 
    \]
    is a nondegenerative simplicial surjection.

\item[(d)] Let $\iota \colon \Z^{n-1} \lra \Z^n$ be a map such that $\pi \circ \iota = \Id$. 
Then $\iota$ induces a simplicial map $\iota\colon X(\Z^{n-1})\lra X(\Z^n)$ such that $\Im (\iota ) \subset  \Lk_{X(\Z^n)}(e_1) $ and whose restriction 
\[
\iota \colon \Lk_{X(\Z^{n-1})}(\pi (\sigma)) \lra \Lk_{X(\Z^n)}(e_1\sigma )
\]
is a section of $\pi $.
\item[(e)] Let $\mu \in \{-1,\,1\}^m$. Then 
\[
\Lk_{X(\Z^n)} (\sigma ) = \Lk_{X(\Z^n)} (\mu (\sigma ))
\]
where $\mu (\sigma ):= \mu_1(e_{n-m+1})\ldots \mu_m( e_n)$.\\ 
If $\tau \in \Lk_{X(\Z^n)}(\sigma)$, then $\mu (\tau) \in \Lk_{X(\Z^n)}(\sigma )$ for all $k$-simplices $\tau \in X(\Z^n)$ and all $\mu \in \{-1,\,1\}^{k+1}$.
\end{itemize}

\end{lemma}
\begin{proof}
\begin{itemize}
   \item[(a)] Let $\tau = u_1\ldots u_k \in X(\Z^n)$. Then
    $\tau \in \Lk(\sigma )$ if and only if the gcd of all $(k+m)\times (k+m)$ minors of the matrix $\big( u_1,\dots, u_k,e_{n-m+1},\dots ,e_n  \Big)$ is equal to $1$, which is also equal to the gcd of all $k \times k$ minors of the matrix $\Big( \pi '(u_1),\ldots ,\pi '(u_k) \Big)$ if and only if $\pi '(\tau ) \in X(\Z^{n-m})$.
    
    \item[(b)] Let $\tau \in \Lk(\sigma )$. Then $\tau\sigma \in X(\Z^n)$. Since $A(\sigma )=\sigma$, we have $A(\tau) \in \Lk(\sigma )$ because $A(\tau )\sigma = A(\tau \sigma ) \in X(\Z^n)$. Therefore, $A\colon \Lk (\sigma ) \lra \Lk (\sigma )$ and it is an isomorphism since $A^{-1}$ is its inverse.
    
    \item[(c)] Let $\tau = u_1\ldots u_k \in \Lk(e_1\sigma )$. Then the gcd of all $(k+m+1)\times (k+m+1)$ minors of $\Big(e_1,u_1,\ldots, u_k, e_{n-m+1},\ldots ,e_n\Big)$ is equal to 1. As it is equal to the gcd of all $(k+m)\times (k+m)$ minors of $\Big( \pi(u_1),\ldots ,\pi(u_k),\pi(e_{n-m+1}),\ldots ,\pi(e_n) \Big)$, we have  $\pi (\tau)\in \Lk_{X(\Z^{n-1})}(\pi(\sigma ))$ and $\pi $ induces a simplicial map. Furthermore, for a simplex $\tau' \in \Lk_{X(\Z^{n-1})}(\pi(\sigma ))$, we have $\iota (\tau') \in \Lk(e_1\sigma )$ and $\pi (\iota (\tau')) = \tau '$, where $\iota$ is induced by $\iota (x_1,\ldots ,x_{n-1}) = (0,x_1,\ldots ,x_{n-1})$.
    
    \item[(d)] For every $k$-simplex $\tau \in X(\Z^{n-1})$, the gcd of all $(k+2)\times (k+2)$ minors of $\Big(e_1,\iota(\tau)\Big)$ is equal to the gcd of all $(k+1)\times (k+1)$ minors of $\Big(\tau\Big)$ which is equal to 1. By \Cref{GCD},  $\Im (\iota ) \subset  \Lk_{X(\Z^n)}(e_1) $.
    
    \item[(e)] It follows directly using Lemma~\ref{GCD} and the fact that the gcd's of all $(k+m+1)\times (k+m+1)$ minors of $\Big( \tau, \sigma \Big)$ and of $\Big( \tau ,\mu (\sigma)\Big)$ are equal for all $k$-simplices $\tau \in X(\Z^n)$. The other part of (e) is proved in the same way. \qedhere
\end{itemize}
\end{proof}

The proofs of the remaining statements for links in this subsection  follow the same lines as of the proof of appropriate statements of $X(\Z^n)$ using \Cref{pi-link} for technical adjustments. Therefore statements that are completely analogous  will be stated without proofs.

\begin{lemma}
\label{link connected}
For every $(m-1)$-simplex $\sigma \in X(\Z^n)$, $1\leq m\leq n-2$, the link $\Lk_{X(\Z^n)}(\sigma)$ is connected. 
\end{lemma}

\begin{proof}
Let $u,v\in\Lk(\sigma)$ be two vertices. If $\pi'(u)\ne \pi'(v)$ then by \Cref{connected} there exists a path between $\pi'(u)$ and $\pi'(v)$ in $X(\Z^{n-m})$ and then by \Cref{pi-link} (a) there exists a path between $u$ and $v$ in $\Lk(\sigma)$.
Let $\pi'(u)= \pi'(v)$. Since $m\leq n-2$, there exists a vertex $w$ such that $uw\in\Lk(\sigma)$. By \Cref{pi-link} (a) the simplex $vw$ is also in $\Lk(\sigma)$ so $v-w-v$ is a path in $\Lk(\sigma)$.
\end{proof} 

\begin{lemma}
\label{link 0}
Let $\sigma \in X(\Z^n)$ be an $(n-2)$-simplex. Then $\Lk_{X(\Z^n)}(\sigma )$ is an infinite 0-dimensional simplicial complex.
\end{lemma}
\begin{proof}
Let $\sigma = e_2\ldots e_n$, as in~\eqref{sigma}. The simplicial complex $\Lk(\sigma )$ is 0-dimensional since $X(\Z^n)$ is a pure $(n-1)$-dimensional simplicial complex. Since
\[
v\in \Lk(\sigma ) \text { if and only if } v(1)=\pm 1
\]
the link $\Lk(\sigma )$ is infinite.
\end{proof}

\begin{proposition}
\label{link simply connected}
For $n\geq 4$, $1\leq m \leq n-3$ and $\sigma $ an $(m-1)$-simplex in $X(\Z^n)$, the simplicial complex $\Lk_{X(\Z^n)}(\sigma )$ is simply connected.
\end{proposition}
\begin{proof}
The proof follows the lines of the proof of \Cref{simply connected} and thus we omit it here.
\end{proof}

\begin{proposition}
\label{link}
Let $n\geq 3$, let $\sigma \in X(\Z^n)$ be an $(m-1)$-dimensional simplex, $1\leq m \leq n-2$, and let $0\leq k \leq n-m-2$. Then every reduced homology group $\widetilde{H}_k(\Lk_{X(\Z^n)}(\sigma );\Z)$ is trivial.
\end{proposition}

We fix $\sigma $ as in~\eqref{sigma} throughout.

\begin{lemma}
\label{link granica}
Let $c$ be a $k$-cycle in $C_k(\Lk_{X(\Z^n)}(\sigma );\Z)$, $1\leq k \leq n-m-2$, such that $M(c)=0$. Then $c$ is a boundary.\qed
\end{lemma}



\begin{lemma}
\label{link mcvece0}
Let $c\in C_k(\Lk(\sigma );\Z)$ be a cycle such that $|M(c)|> 0$, $1\leq k \leq n-m-2$. Then there is a $k$-chain $c'$ homologous to $c$ such that $\max (c')\preceq \max (c)$ and $\Omega_{c'}(M(c))< \Omega_{c}(M(c))$.
\end{lemma}

Putting  all these statements together, we  prove \Cref{link}.

\begin{proof}[Proof of Proposition~\ref{link}]
(i) For $k=0$, the result follows from Lemma~\ref{link connected}.

(ii) For $k>0$,  we prove the statement by induction on $n$.\\
Let $n=3$. Then the statement holds by Lemma~\ref{link connected}.\\
For $n\ge 4$, let $c=\sum_{i=1}^m\lambda_i \sigma_i=\sum_{i=1}^m\lambda_i u_0^i\ldots u_k^i\in C_k(\Lk(\sigma );\Z)$ be a $k$-cycle. If $M(c)=0$, then by \Cref{link granica}, $c$ is a boundary. 

If $|M(c)|>0$, then by Lemma~\ref{link mcvece0}, the cycle $c$ is homologous to a $k$-cycle $c'$ for which $\max (c')\preceq \max (c)$ and $\Omega_{c'}(M(c))<\Omega_{c}(M(c))$. Iterating this process first we find a cycle $\bar c$ homologous to $c $ with $\max (\bar c)\preceq \max (c)$ and $\Omega_{\bar c}(M(c))=0$ and then a $k$-cycle $\hat c$ which is homologous to $c$ and such that $M(\hat c)=0$. Since $M(\hat c)=0$,  by \Cref{link granica} the cycle $\hat c$ is a boundary  and so is $c$.
\end{proof}

To prove~\Cref{link mcvece0} we construct a $k$-chain $c'$ using the same idea and notation as in \Cref{mcvece0}.
The analogue of the matrix $A$ in \Cref{mcvece0} is the $n\times n$ matrix $A$ such that its first column is the vector $u$, the last $m$ columns are $e_{n-m+1}, \ldots , e_n$ so that by \Cref{pi-link}(b)), $A$ maps  $\Lk(\sigma )$ to itself  and $\det A=1$. 
Then the $(k-1)$-cycle $c_u$ is mapped to a $(k-1)$-cycle $A^{-1}(c_u)=\sum_{i\in S}\lambda_i v_1^i\ldots v_k^i$ in $C_{k-1}(\Lk(\sigma );\Z)$.

Let $\pi$ and $\hat\iota$ be maps in \Cref{pi}. By \Cref{pi-link}(c), $\pi(A^{-1}(c_u))$ is a $(k-1)$-cycle in $C_{k-1}(\Lk_{X(\Z^{n-1})}(\pi (\sigma ));\Z)$ and by \Cref{pi-link}(d), $\hat\iota (\Lk_{X(\Z^{n-1})}(\pi (\sigma )) \subset \Lk(\sigma )$.
Recall that  $\hat\iota(\pi(v_j^i))=v_j^i$ if and only if $0\le u_j^i(1)<|u(1)|$ and that\[
L(c) := \max_{i\in S}
| \{ j \text{ }|\text{ } \hat\iota (\pi (v_j^i))\neq v_j^i \in A^{-1}(c_u) \}| . 
\]

\begin{proof}[Proof of \Cref{link mcvece0}]
{\bf (i) Let $L(c)=0$.}
Then $\hat\iota(\pi(v))=v$ for all $v\in A^{-1}(c_u)$.
Since $e_1A^{-1}(c_u)$ is a $k$-chain in $C_{k}(\Lk(\sigma );\Z)$, by Lemma~\ref{pi-link}(c) $\pi(A^{-1}(c_u))$ is a $(k-1)$-chain in $C_{k-1}(\Lk_{X(\Z^{n-1})}(\pi (\sigma ));\Z)$. As $c_u$ is a cycle, so it is $\pi(A^{-1}(c_u))$. By the induction hypothesis, there exists a $k$-chain $d\in C_k(\Lk_{X(\Z^{n-1})}(\pi (\sigma ));\Z)$ such that $\partial(d)=\pi(A^{-1}(c_u))$. Then $A^{-1}(c_u)=\hat\iota(\pi(A^{-1}(c_u)))=\hat\iota(\partial (d))=\partial (\hat\iota(d))$, and hence by Lemma~\ref{pi-link}(d) $e_1\hat\iota(d) \in \Lk(\sigma )$ and
\[
\partial(e_1\hat\iota(d))=\hat\iota(d)-e_1\partial(\hat\iota(d))=\hat\iota(d)-e_1 A^{-1}(c_u)=\hat\iota(d)- A^{-1}(u c_u).
\]
Therefore, the homology class of $u c_u- A(\hat\iota(d))$ is trivial and we have constructed a chain $c'=A(\hat\iota(d))+\sum_{i\not\in S} \lambda_i u_0^i\ldots u_k^i$ such that $c$ is homologous to $c'$. Note that $A(\hat \iota (v)) \prec u$ for all vertices $v\in X(\Z^{n-1})$ 
and therefore $\max (A(\hat \iota (d)))\prec u $. Because $\max (A(\hat\iota(d)))\prec \max (u c_u)=\max (c)$, we get $\Omega_{c'}(M(c))<\Omega_{c}(M(c))$.
 
{\bf (ii) Let $L(c)\geq 1$} This case is a complete analogue of \Cref{mcvece0} so we do not include it here.
\end{proof}

\vspace{4mm}
\subsection{Proof of Theorem~\ref{main:z}}~\par

\vspace{2mm}
\noindent{\it Proof of (a).}
If $n=2$, then by~\Cref{connected} $X(\Z^2)$ is a connected 1-dimensional simplicial complex with countably infinitely many 1-simplices. To each simplex $u_0u_1$, we associate a cycle
\[
c(u_0u_1):=\sum_{\mu \in \{-1,\, 1\}^{2}} (-1)^{|\mu |} \mu (u_0u_1).
\]
Therefore $X(\Z^2)$ is homotopy equivalent to a countably infinite wedge of 1-dimensional spheres.

By \Cref{simply connected}, for $n\geq 3$ the simplicial complex  $X(\Z^n)$ is simply connected and $(n-1)$-dimensional. Since $\widetilde{H}_k(X(\Z^n);\Z)=0$ for $k=0,\ldots ,n-2$, the simplicial complex $X(\Z^n)$ is $(n-2)$-connected and $H_{n-1}(X(\Z^n);\Z)$ is a free abelian group on the set of generators $\{g_{\alpha}\,|\,\alpha \in I \}$. For each $(n-1)$-simplex $\sigma \in X(\Z^n)$, let
\[
c(\sigma):=\sum_{\mu \in \{-1,\, 1\}^{n}} (-1)^{|\mu |} \mu (\sigma).
\]
By \Cref{n-1 cycle} $c(\sigma)$ is a cycle, representing a nontrivial homology class. Note that $c(\sigma)$'s are  linearly independent with the only relation $c(\sigma )= c(\mu (\sigma ))$.
Since $X(\Z^n)$ has countably infinitely many top simplices, there are countably infinitely many generators of $H_{n-1}(X(\Z^n);\Z)$.
By the Hurewicz theorem, $\pi_{n-1}(X(\Z^n))\cong H_{n-1}(X(\Z^n);\Z)$
proving that $X(\Z^n)$ is homotopy equivalent to a countably infinite wedge of $(n-1)$-spheres. 

 \vspace{2mm}
\noindent{\it Proof of (b).}
For $0\leq m \leq n-2$, consider the $m$-simplex $\sigma=e_{n-m}e_{n-m+1}\ldots e_n \in X(\Z^n)$.

If $m=n-2$,  then by \Cref{link 0} the link $\Lk (\sigma )$ is a discrete countable simplicial complex and the statement holds.

If $m=n-3$, then by \Cref{link connected} the link $\Lk (\sigma )$ is a connected infinite 1-dimensional simplicial complex. 
If $0\leq m\leq n-4$, then by \Cref{link simply connected} the link $\Lk (\sigma )$ is simply connected  and by \Cref{link}  the link $\Lk (\sigma )$ is $(n-m-3)$-connected. 
 
For each top dimensional simplex $\tau \in \Lk (\sigma )$, by \Cref{pi-link}(e) the chain  $c(\tau )\in C_{n-m-2}(\Lk(\sigma);\Z)$  is a cycle. There are countably infinitely many of them which are linearly independent,  proving that $\Lk (\sigma )$ is homotopy equivalent to a countably infinite wedge of $(n-m-2)$-spheres for $0\leq m \leq n-3$.

\vspace{2mm}
\noindent{\it Proof of (c).} Consider the map that sends a nonzero vector $v\in X(\Z^n)$ to the line throw origin spanned by $v$. This map
extends to a simplicial surjection 
\begin{equation}\label{phi}
\phi \colon X(\Z^n) \lra K(\Z^n)
\end{equation}
which is nondegenerate. The map $\phi$ is  $2^k$ to $1$ on $(k-1)$-simplices.

Furthermore, there is a map 
\begin{equation}\label{psi}
\psi \colon K(\Z^n) \lra X(\Z^n)
\end{equation}
defined on vertices of $K(\Z^n)$ by mapping each line to an unimodular
element from that line and extending it to a simplicial map. A map
$\psi$ depends on the choice of unimodular elements. It is
nondegenerate and it is a section of $\phi$, that is, $\phi \psi =
\Id_{K(\Z^n)}$. Each such $\psi$ embeds $K(\Z^n)$ as a full subcomplex
of $X(\Z^n)$ which is a retract of $X(\Z^n)$. 

From \eqref{phi} and \eqref{psi} we have well defined maps 
\[
\phi \colon X(\Z^n) \lra K(\Z^n)  \text{  and  } \psi \colon K(\Z^n) \lra X(\Z^n),
\]
which satisfy the following relations
\[
\phi (\psi (\sigma)) = \sigma \text{  and  } \psi (\phi (\sigma)) = \mu (\sigma)
\]
for some $\mu \in \{-1,\, 1\}^n$.

Denote by $L_i, i=1,\ldots n$, the lines through $e_i$ and by $L_0$ the line through $(1,1,\ldots ,1)$. Then $\sigma_i := L_0L_1\ldots \hat{L_i}\ldots L_n, i=0,1,\ldots ,n$ is a  $(n-1)$-simplex in $K(\Z^n)$ and 
\begin{equation}
\label{top cycle}
c:= \sum_{i=0}^n (-1)^{i}\sigma_i
\end{equation}
is an $(n-1)$-cycle. 

For each $n\times n$ matrix $A$ with determinant $\pm 1$,  the chain $A(c)$ is a cycle and there are countably infinitely many of them which are linearly independent. Thus $K(\Z^n)$ is homotopy equivalent to a wedge of countably infinite spheres $S^{n-1}$.

\vspace{2mm}
\noindent{\it Proof of (d).}
Let $\sigma $ be an $m$-simplex in $K(\Z^n)$, $0\leq m \leq n-2$.
The maps $\phi $ and $\psi$ preserve links, that is,
\[
\phi (\Lk_{X(\Z^n)}(\sigma )) \subset \Lk_{K(\Z^n)}(\phi (\sigma )) \text{ and }\psi(\Lk_{K(\Z^n)}(\sigma )) \subset \Lk_{X(\Z^n)}(\psi (\sigma )).
\]
Since $\phi \circ \psi = \Id$, the link $\Lk_{K(\Z^n)}(\sigma )$ is a retract of $\Lk_{X(\Z^n)}(\psi (\sigma ))$. Because $\Lk_{X(\Z^n)}(\psi (\sigma ))$ is $(n-m-3)$-connected, so it is $\Lk_{K(\Z^n)}(\sigma )$.

Analogously as in part (c), we prove that $H_{n-m-2}(\Lk_{X(\Z^n)}(\sigma )$ is a free abelian group with countably infinitely many generators. Let $\sigma = L_{n-m}L_{n-m+1}\ldots L_n$ and let $\sigma_i := L_0L_1\ldots \hat{L_i} \ldots L_{n-m-1}$. Then
\[
c:= \sum_{i=0}^{n-m-1}\sigma_i \in C_{n-m-2}(\Lk_{K(\Z^n)}(\sigma );\Z)
\]
is a cycle.

For each $n\times n$ matrix $A$ with the last $m+1$ columns $e_{n-m},\ldots ,e_n$ and with determinant  $\pm 1$, the chain $A(c)$ is a cycle in $\Lk( \sigma )$ and therefore there are countably infinitely many of top dimensional cycles which are linearly independent. Thus $\Lk_{K(\Z^n)}(\sigma )$ is homotopy equivalent to a wedge of a countably infinitely many spheres $S^{n-m-2}$.
\qed

\section{Universal Simplicial Complexes and Buchstaber Invariant}
\label{sec:invariant}
Let $K$ be a finite simplicial complex on vertex set $[m]=\{1,\ldots, m\}$ 
and $(X, A)$, $A \subset X$ be a pair of topological spaces. For every $\sigma \in K$, define \[(X, A)^{\sigma}:=\{(x_1, x_2, \ldots, x_m)\in X^m\, |\, x_i \in A \mbox{\, if\, } i\notin \sigma \}.\] 

\begin{definition}\label{polprodK} The polyhedral $K$ product of $(X, A)$ is a topological space \[(X,A)^K:=\bigcup_{\sigma\in K} (X, A)^\sigma.\]
\end{definition} 

The polyhedral products $(D^2, S^1)^K$ and $(I, S^0)^K$ are of special interests in toric topology and they are called moment-angle complex $\mathcal{Z}_K$ and real moment-angle complex $ \mathcal{R}_K$, respectively. $\mathcal{Z}_K$  has the natural coordinatewise action of the torus $T^m$ while $ \mathcal{R}_K$ has the natural coordinatewise action of the real torus $\F_2^m$.

\begin{definition}\label{bi} A \textit{complex Buchstaber invariant} $s(K)$ of $K$ is a maximal dimension of a toric subgroup of $T^m$ acting freely on $\mathcal{Z}_K$.

A \textit{real Buchstaber invariant} $s_{\mathbb{R}} (K)$ of $K$ is a maximal rank of a subgroup of $\F_2^m$ acting freely on $ \mathcal{R}_K$.
\end{definition}

It is an open problem in toric topology to find a combinatorial description of $s(K)$ and $s_{\mathbb{R}} (K)$. The Buchstaber invariants are closely related to universal simplicial complexes $K (\mathbf{\mathbb{Z}}^n)$ and $K (\mathbf{\mathbb{Z}}_2^n)$. Following~\cite[Section 2]{Ayze1}, we recall few their properties needed for our work.

\begin{proposition}\label{bin} A \textit{complex Buchstaber invariant}  of $K$ is the integer $s(K)=m-r$, where $r$ is the least integer such that there is a nondegenerate simplicial map  
\[ 
f\colon K \rightarrow K
	(\mathbf{\mathbb{Z}}^r) .
	\] 	
A \textit{real Buchstaber invariant}  of $K$ is the integer $s_{\F_2} (K)=m-r$, where $r$ is the least integer such that there is a nondegenerate simplicial map 
	\[ f\colon K \rightarrow K (\F_2^r) . 
	\] \qed
\end{proposition}
Recall from \eqref{phi} and \eqref{psi} that there are nondegenerate maps $\phi \colon \xkn \lra \kkn$ and $\psi \colon \kkn \lra \xkn$. Thus $X(\Z^r)$ can be used instead of $K(\Z^r)$ in Proposition~\ref{bin} to compute a complex Buchstaber invariant. Note that for the real case  $X(\F_2^r)=K(\F_2^r)$.

Finding explicit values of the Buchstaber invariants for a given simplicial complex $K$ is an important open problem in toric topology and in general only some estimates of these invariants are known. A sequence $\{L^i\}_{i\in\N}$ of simplicial complexes is called an \textit{increasing sequence} if for every $i$ and $j$ such that $i<j$ there is a nondegenerate simplicial map $f\colon L^i\rightarrow L^j$. Observe that Proposition~\ref{bin} can be used to define the Buchstaber invariant $s_{{\{L^n\}}} (K)$
for any increasing sequence of simplicial complexes  $\{L^n\}$ instead of $\{K(\F_2^n)\}$ or $\{K({\mathbb{Z}}^n)\}$. For the sequence of standard simplicial complexes $\{\Delta^i\}$, the Buchstaber invariants are $m-\gamma(K)$, where $\gamma (K)$ is the classical chromatic number of a simplicial complex. In particular, for a prime number $p$ the family $\{K(\F_p^n)\}$ can be considered which allows for the following definition.

\begin{definition}
Let $K$ be a simplicial complex. A \emph{mod $p$ Buchstaber invariant} $s_{\F_p}(K)$ is given by
\[ 
s_{\F_p}(K)=m-r
\]
where $r$ is the least integer such that there is a nondegenerate simplicial map $f\colon K \rightarrow K(\F_p^r)$.
\end{definition}

Note that $s_{\mathbb{R}} (K)$ is mod $2$ Buchtaber invariant, that is, $s_{\F_2} (K)$ in this notation. 

\begin{remark}
\label{ref2}
A topological description of $s_{\F_p} (K)$ can be given as a maximal rank of subgroup of $\Z_p^m$  acting freely on  the polyhedral product $(\mathrm{Cone}( \Z_p), \Z_p)^K\subset D^{2m}$, where $\Z_p$ is considered as the set of $p$-th roots of unity. Unlike in the cases of $\F_2^n$ and $\mathbb{Z}$, $(\mathrm{Cone}( \Z_p), \Z_p)^K$ for $p>2$ fails to be a manifold when $K$ is a triangulation of a sphere.
\end{remark}

Ayzenberg in \cite {Ayze} and Erokhovets in \cite{MR3482595} studied nondegenerate maps between various increasing sequences of simplicial complexes to establish estimates of Buchstaber invariants in terms of combinatorial invariants of simplicial complexes. They obtained two interesting results.

\begin{proposition}[\cite{Ayze1}, Proposition 2] \label{ay1} Let $\{L^i\}$ and $\{M^i\}$ be  two increasing sequences of simplicial complexes such that for each $i$ there is a nondegenerate map $f_i\colon L^i\rightarrow M^i$. Then $$s_{\{L^i\}} (K)\leq s_{\{M^i\}} (K).$$\qed
\end{proposition}

\begin{proposition}[\cite{Ayze1}, Proposition 4] \label{ay2} Let $\{L^i\}$  be  an increasing sequence of simplicial complexes and let $K_1$ and $K_2$ be simplicial complexes on vertex sets $[m_1]$ and $[m_2]$, respectively. Suppose that there exists a nondegenerate map $f\colon K_1 \rightarrow K_2$. Then \[
s_{\{L^i\}} (K_1)-s_{\{L^i\}} (K_2)\geq m_1-m_2.
\]\qed
\end{proposition}
Recall that the \textit{chromatic number} of a simplicial complex $K$ is the minimal number $\gamma (K)$ such that there exists a nondegenerate map $f \colon K \rightarrow \Delta^{\gamma (K)-1}$.
 
 \begin{corollary}
 Let $K$ be a simplicial complex on $m$ vertices. Then
 \begin{equation}
\label{in:pg}
m-\gamma (K) \leq s (K)\leq s_{\F_p} (K)\leq m-\dim{K}-1.
\end{equation}
 \end{corollary}
\begin{proof}
There are an inclusion of $\Delta^{n-1}$ into $K({\mathbb{Z}}^n)$ and a modulo $p$ reduction map from
$K(\mathbb{Z}^n)$ to $K(\F_p^n)$ which is nondegenerate for any prime number $p$. The statement now follows by Proposition~\ref{ay1}. 
\end{proof}

\begin{proposition}
For every simplicial complex $K$, there is an infinite set of prime numbers denoted by $P (K)$ such that for all primes $p$, $q\in P (K)$ $$s_{\F_p} (K)=s_{\F_q} (K).$$
\end{proposition}

\begin{proof}
Using the Pigeonhole principle, relation~\eqref{in:pg} implies the claim since the set of values of $s_{\F_p} (K)$ is a finite set, which contradicts the fact that the set of prime numbers is countable.
\end{proof}
As an immediate corollary of Propositions \ref{ay1} and \ref{ay2}, in the same fashion as it was done for $p=2$ in \cite{Ayze1}, the following result holds.

\begin{corollary} Let $K$ be a simplicial complex on vertex set $[m]$ and $p$ a prime number. Then the following inequality holds
	\begin{equation}
s (K)\leq s_{\F_p} (K)\leq m-\lceil \log_p \left((p-1)\gamma (K)+1 \right)\rceil.
	\end{equation}
\end{corollary}

\begin{proof} Only the second inequality needs to be proved. Let $r$ be a minimal integer such that there is a nondegenerate map $f\colon K\rightarrow K(\F_p^r)$. There is a nondegenerate inclusion $e \colon K(\F_p^r)\rightarrow \Delta^{\frac{p^r-1}{p-1}-1}$ as $ K(\F_p^r)$ has $\frac{p^r-1}{p-1}$ vertices. Thus, the composition $ e\circ  f$ is a nongenerate map from  $K$ to $\Delta^{\frac{p^r-1}{p-1}-1}$. However, $\gamma (K)-1$ is the minimal dimension of a simplex into which $K$ can be mapped nondegeratively so 
\[
\gamma (K)\leq {\frac{p^r-1}{p-1}}.
\]
The desired inequality follows from the inequality above.  
\end{proof} 

Using the same argument with appropriate modifications as in the Ayzenberg's proof of \cite[Theorem~1]{Ayze1}, we  deduce the following result.

\begin{proposition} 
\label{prop}
For any simple graph $\Gamma$,  
$$
s_{\F_p} (\Gamma)=m-\left \lceil \log_p ((p-1) \gamma (\Gamma)+1)\right \rceil.
$$\qed

\end{proposition}

\begin{corollary}\label{graf}
Let $\Gamma$ be a simple graph on $m$ vertices and $p$ a prime number. Let $k$ be a natural number such that
\[
p^{k-2}+p^{k-3}+\ldots +1<\gamma (\Gamma) \leq p^{k-1}+p^{k-2}+\ldots +1.
\]
Then
\[
s_{\F_p}(\Gamma)=m-k.
\]
In particular, if $\Gamma$ is not a discrete graph and $\gamma (\Gamma )\leq p+1$ then $s_{\F_p}(\Gamma)=m-2. $  
\end{corollary}
\begin{proof}
For assumed $k$, we have
\[
p^{k-1} < (p-1)\gamma (\Gamma) +1 \leq p^k.
\]
Applying $\log_p$ to this inequality implies that
\[
\left \lceil \log_p ((p-1) \gamma (\Gamma)+1)\right \rceil =k.
\]
Proposition~\ref{prop} now finishes the proof.
\end{proof}

 Corollary \ref{graf} implies that a simple graph $\Gamma$ can be mapped by a nondegenerate map in $K(\mathbb{F}_p^k)$ when $k$ is such that
\[
p^{k-2}+p^{k-3}+\ldots +1<\gamma (\Gamma) \leq p^{k-1}+p^{k-2}+\ldots +1.
\]
It is easy to realise a map like that. Let $f \colon V(\Gamma) \rightarrow [ \gamma(\Gamma)]$ be a proper colouring of the vertices of $\Gamma$. Since the number of vertices of $K(\mathbb{F}_p^k)$ is $p^{k-1}+\ldots +1$, there is an injection $i \colon  [ \gamma(\Gamma)] \rightarrow K(\mathbb{F}_p^k)$. The composition $i\circ f\colon V(\Gamma )\rightarrow K(\mathbb{F}_p^k)$ has an extension to a simplicial map $\overline{i\circ f}\colon \Gamma \rightarrow K(\mathbb{F}_p^k)$ because the $1$-skeleton of $K(\mathbb{F}_p^k)$ is a complete graph.

\begin{proposition} \label{my1} Let $p$, $q$ be prime numbers and let $m$ and $n$ be positive integers. Let $f$ be a nondegenerate map $f\colon K(\F_p^m)\rightarrow K(\F_q^n)$. Then $f$ is an injection.
\end{proposition}

\begin{proof} By dimensional reasons, if $f$ is a nondegenerate map $m$ must be less than $n$. Observe that the $1$-skeletons $K^{(1)}(\F_p^m)$ and  $K^{(1)}(\F_q^n)$ are complete graphs, so nondegeneracy of $f$ implies that $f$ is an injection on the set of vertices of $K(\F_p^m)$. It follows that $f$ is an injection since it is a simplicial map.
\end{proof}

An instant consequence of Proposition \ref{my1} is the following fact.

\begin{corollary}\label{my2}
If there is a nondegenerate map $f\colon K(\F_p^m)\rightarrow K(\F_q^n)$,  then 
    \begin{eqnarray}
    m &\leq & n,\\
f_i \left(  K(\F_p^m)\right)& \leq & f_i \left(  K(\F_q^n)\right)  \text{\,\, for all\,\, } i.
\end{eqnarray}
\qed
\end{corollary}
 
For any two prime numbers $p$ and $q$, define a function $\varsigma_{p, q}(n)$ that assigns to each positive integer $n$ the minimal integer $r$ such that there is a nondegenerate map from $K(\F_p^n)$ to $K(\F_q^r)$. Similarly, for a prime number $p$ define a function $\vartheta_p(n)$ such that $\vartheta_p (n)$ is the minimal integer $r$ such that there is a nondegenerate map from $K(\F_p^n)$ to $K(\mathbb{Z}^r)$.

\begin{proposition} The functions $\varsigma_{p, q}(n)$ and $\vartheta_{p}(n)$ are increasing. 
\end{proposition}

\begin{proof} The statement follows from the inclusion of $K(\mathbf{k}^n)$ into $K(\mathbf{k}^{n+1})$ and Proposition \ref{ay1}.
\end{proof}

\begin{theorem}\label{bounds1}
$$\left\lceil \log_q \left(\frac{(q-1)(p^n-1)}{p-1}+1\right) \right\rceil \leq \varsigma_{p, q}(n)\leq \frac{p^n-1}{p-1}.$$
\end{theorem}
\begin{proof}
Corollary \ref{my2} implies the following relation \begin{equation}\label{ob1}
\frac{p^n-1}{p-1}\leq \frac{q^{\varsigma_{p,q} (n)}-1}{q-1}.
\end{equation} 
On the other hand, $K(\F_p^n)$ includes into $K({\F_q}^{\frac{p^n-1}{p-1}})$ via the composition of the inclusions 
\begin{equation}\label{ob2}
K(\F_p^n)\xhookrightarrow{} \Delta^{\frac{p^n-1}{p-1}-1}\xhookrightarrow{} K({\F_q}^{\frac{p^n-1}{p-1}}). 
\end{equation}

From \eqref{ob1} and \eqref{ob2}, we deduce the desired inequalities.
\end{proof}

By Theorem~\ref{bounds1}, the function $\varsigma_{p, q}(n)$ is bounded above with $ \frac{p^n-1}{p-1}$ implying that it is well defined.

For any positive integer $n$, there is a composition of nondegenerate maps 
\[
K(\F_p^n)\rightarrow K(\Z^{\vartheta_p (n)})\rightarrow K(\F_q^{\vartheta_{p}(n)})
\]implying the following statement.

\begin{proposition}
\label{estimate} For every positive integer $n$ and any two prime numbers $p$ and $q$, the following inequality holds $$\vartheta_p (n)\geq \varsigma_{p, q} (n).$$\qed
\end{proposition}
 Using  Proposition~\ref{estimate}, the lower bound of $\varsigma_{p, q} (n)$ given in Theorem~\ref{bounds1}  is also a lower bound of $\vartheta_p (n)$. The largest lower bound is reached for $q=2$, so 
\[
\vartheta_p (n)\geq \left\lceil \log_2 \left(\frac{p^n-1}{p-1}+1\right) \right\rceil.
\]
In a similar way, using~\eqref{ob2}, we have
\[
\vartheta_p (n)\leq \frac{p^n-1}{p-1}.
\]

\section{Bhargava's generalized factorial function}
A fundamental interpretation of the factorial function is the number of all permutations of a set with $n$ elements. Its countless appearance in numerous combinatorial formulas such as binomial coefficients, Stirling numbers, Catalan numbers, etc. makes the factorial function one of the principal functions in combinatorics. The factorial function has important appearances in number theory and mathematical analysis as well motivating different generalisations to suit variety of number-theoretic, ring-theoretic and combinatorial problems.  

Bhargava~\cite{Bhar} considered one such generalisation of the classical factorial function on $\mathbb{Z}$ that satisfies the corresponding analogues of classical theorems of elementary number theory: the binomial coefficient theorem, the theorem about the greatest common factor of the set of all integer values of a primitive polynomial (that is, a polynomial whose coefficients are relatively prime numbers), the theorem on the number of polynomial functions from $\mathbb{Z}$  to $\mathbb{Z}_n$, and the theorem about the product of pairwise differences of any $n+1$ integers. 

Let $S$ be a subset of $\mathbb{Z}$ and $p$ be a prime number. 

\begin{definition} A $p$-ordering on $S$ is a sequence of elements ${(a_i)}_{i=0}^{\infty}$ formed by induction according to the following rules:

\begin{enumerate}
\item[(i)] Choose an element $a_0\in S$ arbitrary.

\item[(ii)] For any positive integer $n$, the integer $a_n$ is chosen so that it minimises the highest powers of $p$ dividing $(a_n-a_{n-1}) \cdots (a_n-a_1)(a_n-a_0)$. 
\end{enumerate}
\end{definition}

For a $p$-ordering on $S$ and $k$ a positive integer, the number $\nu_k (S, p)$ is defined to be the highest power of $p$ dividing $(a_k-a_{k-1}) \cdots (a_k-a_1)(a_k-a_0)$. There are many distinct $p$-orderings on $S$, but the numbers $\nu_k (S, p)$ are independent of a choice of $p$-ordering. Therefore, $\nu_k (S, p)$ are invariant of the subset $S$, as it was proved by Bhargava in \cite{Bhar}.
Using this observation, Bhargava defined the generalised factorial function for any subset $S$ as 
\[
{k!}_S=\prod_p \nu_k (S, p) 
\]
where the product is taken over all prime numbers.

For more details about the generalised factorial function and its number-theoretical properties see the aforementioned Bhargawa's paper. Motivated by the classical combinatorial description of the standard factorial function, Bhargawa~\cite{Bhar} formulated the following problem.

\begin{problem}[\cite{Bhar}, Question 27] \label{q27} For a subset of $S$, is there a natural combinatorial interpretation of the number ${k!}_S$.
\end{problem}

We solve Problem~\ref{q27} when $S$ is the set of powers of a prime number $p$.
Bhargava showed in~\cite{Bhar} that if $S$ is a geometric progression in $\mathbb Z$ with common ratio $q$ and first term $a$ then
\begin{equation}
\label{fact}
k!_S=a^k(q^k-1)(q^k-q)\cdots (q^k-q^{k-1}).
\end{equation}

\begin{proposition} For the set $S$ of powers of a prime number $p$, the generalised factorial function ${i!}_S$ is the product of $i!$ and $f_{i-1} (X (\mathbb{F}_p^i))$.
\end{proposition}
\begin{proof}
It is easy to verify that 1, $p, p^2, p^3,\ldots$ forms an $l$-ordering of $S$ for all primes $l$. Using identity~\eqref{fact} and Theorem~\ref{month}, 
\[
{i!}_S=(p^i-p^{i-1}) \cdots (p^i-p^0)= i!f_{i-1}(X(\F_p^i)). \qedhere
\]
\end{proof}

 We remark that in this case there is another interpretation of ${i!}_S$ as the number of $(i-1)$-simplices of a simplicial complex of ordered unimodular sequences in $\mathbb{F}_p^i$,   that is, the number of unimodular sequences in $\mathbb{F}_p^i$ of length $i$. This complex is a special case of a simplicial complex described in~\cite[Definition~2.3]{MR586429}.

Bhargava in \cite{Bhar} established that for $T\subset S \subset \mathbb{Z}$,  generalised factorial function ${k!}_S$ divides ${k!}_T$ for all positive integers $k$.
We give a combinatorial interpretation of $\frac{k!_T}{k!_S}$ when $S=\mathbb{Z}$, $T=\{1, p,\dots, p^k, \dots \}$ and $p$ is a prime number. 

\begin{lemma}
\label{interpretation}
Let $p$ be a prime number, $S=\mathbb{Z}$ and $T=\{1, p,\dots, p^k, \dots \}$.
Then $$\frac{k!_T}{k!_S}=f_{k-1} (X(\mathbb{F}_p^k)).$$
\end{lemma}
\begin{proof}
By Lemma~\ref{month}, 
\[
 f_{k-1} (X(\mathbb{F}_p^k))= \frac{(p^k-p^{k-1}) \cdots (p^k-p^0)}{k!}=\frac{k!_T}{k!_S}.
\]
\end{proof}
\Cref{interpretation} also gives a proof of divisibility of $(p^k-p^{k-1}) \cdots (p^k-p^0)$ by $k!$. 
It would be interesting to find a purely combinatorial proof of Bhargava's result on divisibility in general.

\begin{center}\textmd{\textbf{Acknowledgements} }
\end{center}

\medskip
The first author is grateful to the Vietnam Institute for Advanced Studies in Mathematics in Hanoi for the
hospitality, excellent working environment and generous financial support while accomplishing this research. The first, third and fourth authors were partially supported by the bilateral project ``Discrete Morse theory and its Applications" funded by the Ministry for Education and Science of the Republic of Serbia and the Ministry of Education, Science and Sport of the Republic of Slovenia. 
The third author was supported by the Slovenian Research Agency program P1-0292 and the grants J1-8131, N1-0064,
and N1-0083.

\bibliographystyle{amsplain}
\bibliography{mybibliography}

\end{document}